\documentclass[11pt]{amsart}
\usepackage{graphicx}
\usepackage[space]{grffile}
\usepackage[T1]{fontenc}
\usepackage{color}
\usepackage[colorlinks,linkcolor=blue,citecolor=blue, pdfstartview=FitH]{hyperref}
\usepackage{backref}
\usepackage{amssymb}
\usepackage{epsf}
\usepackage{caption}
\usepackage{slashed}

\def\p{\partial}
\def\o{\overline}
\def\b{\bar}

\def\mb{\mathbb}
\def\mc{\mathcal}
\def\mr{\mathrm}

\def\n{\nabla}
\def\wt{\widetilde}

\theoremstyle{plain}
\newtheorem{theorem}{Theorem}[section]
\theoremstyle{plain}
\newtheorem{definition}[theorem]{Definition}
\theoremstyle{plain}

\theoremstyle{remark}
\newtheorem{remark}[theorem]{Remark}
\theoremstyle{remark}

\theoremstyle{remark}
\newtheorem{question}[theorem]{Question}
\theoremstyle{plain}
\newtheorem{proposition}[theorem]{Proposition}
\numberwithin{equation}{section}
\theoremstyle{plain}
\newtheorem{example}[theorem]{Example}
\theoremstyle{plain}
\newtheorem{corollary}[theorem]{Corollary}

\usepackage{fancyhdr}
\makeatletter
\@namedef{subjclassname@2020}{\textup{2020} Mathematics Subject Classification}
\makeatother

\begin{document}

\title{Positivity of Schur forms for strongly decomposably positive vector bundles}

\author{Xueyuan Wan}
\address{Xueyuan Wan: Mathematical Science Research Center, Chongqing University of Technology, Chongqing 400054, China.}
\email{xwan@cqut.edu.cn}

\date{\today}

\begin{abstract}

In this paper, we define two types of strongly decomposable positivity, which serve as generalizations of (dual) Nakano positivity and are stronger than the decomposable positivity introduced by S. Finski. We provide the criteria for strongly decomposable positivity of type I and type II and prove that the Schur forms of a strongly decomposable positive vector bundle of type I are weakly positive, while the Schur forms of a strongly decomposable positive vector bundle of type II are positive. These answer a question of Griffiths affirmatively for strongly decomposably positive vector bundles. Consequently, we present an algebraic proof of the positivity of Schur forms for (dual) Nakano positive vector bundles, which was initially proven by S. Finski.

\end{abstract}


  \thanks{Research of Xueyuan Wan is partially supported by the National Natural Science Foundation of China (Grant No. 12101093) and the Natural Science Foundation of Chongqing (Grant No. CSTB2022NSCQ-JQX0008), the Scientific Research Foundation of the Chongqing University of Technology.}
\maketitle

\tableofcontents

\section{Introduction}

Let $(E,h^E)$ be a Hermitian holomorphic vector bundle of rank $r$ over a complex manifold $X$ of dimension $n$. The Chern forms $c_i(E,h^E)$ of degree $2i$, $0\leq i\leq r$, and the total Chern form $c(E,h^E)$ are defined by 
\begin{equation*}
 c(E,h^E):= \sum_{i=0}^rc_i(E,h^E):=\det\left(\mr{Id}_E+\frac{\sqrt{-1}}{2\pi}R^E\right),
\end{equation*}
where $R^E\in A^{1,1}(X,\mr{End}(E))$ denotes the Chern curvature of $(E,h^E)$. For any $k\in \mb{N}$ with $1\leq k\leq n$,
 let $\Lambda(k, r)$ be the set of all the partitions of $k$ by non-negative integers less than or equal to $r$, i.e., any element $\lambda=(\lambda_1,\cdots,\lambda_k) \in \Lambda(k, r)$ satisfying
$$
r \geqslant \lambda_1 \geqslant \lambda_2 \geqslant \cdots \geqslant \lambda_k \geqslant 0\text{ and }|\lambda|=\sum_{i=1}^k \lambda_i=k.
$$
 Each partition $\lambda \in \Lambda(k, r)$ gives rise to a Schur form by 
\begin{equation*}
  P_\lambda(c(E,h^E)):=\det(c_{\lambda_i-i+j}(E,h^E))_{1\leq i,j\leq k},
\end{equation*}
which is a closed real $(k,k)$-form.
The Schur forms contain the Chern forms and the signed Segre forms as special examples, e.g., 
\begin{equation*}
  P_{(k,0,\cdots,0)}(c(E,h^E))=c_k(E,h^E)
\end{equation*}
and 
\begin{equation*}
   P_{(1,\cdots,1,0,\cdots,0)}(c(E,h^E))=(-1)^ks_k(E,h^E).
\end{equation*}

In \cite[Page 129, Conjecture (0.7)]{Griffiths}, Griffiths conjectured the numerical positivity of Griffiths positive vector bundles (see \eqref{Griffiths-positive} for a definition), that is, if $(E,h^E)$ is a Griffiths positive vector bundle, then 
\begin{equation}\label{num-pos}
  \int_V P(c_1,\cdots,c_s)>0
\end{equation}
where $P(c_1,\cdots,c_s)$ is a positive polynomial in the Chern classes $c_1,\cdots, c_s$ of any quotient bundle $Q$ of $E|_V$, $V\subset X$ is any complex analytic subvariety. Bloch-Gieseker \cite{MR297773} proved that all Chern classes of an ample vector bundle satisfy \eqref{num-pos}, Fulton-Lazarsfeld \cite[Theorem I]{FL} extended Bloch-Gieseker's result and proved all Schur polynomials are numerically positive for ample vector bundles. For nef vector bundles over compact K\"ahler manifolds, Demailly-Peternell-Schneider \cite[Theorem 2.5]{MR1257325} proved the numerical semi-positivity of all Schur polynomials. 

Griffiths \cite[Page 247]{Griffiths} also conjectured \eqref{num-pos} holds on the level of the differential forms, which can be reformulated as follows, see \cite[Page 1541, Question of Griffiths]{Fin}.
\begin{question}[Griffiths]\label{Question}
	Let $P \in \mathbb{R}\left[c_1, \ldots, c_r\right]$ be a non-zero non-negative linear combination of Schur polynomials of weighted degree $k$. Are the forms $P\left(c_1(E, h^E), \ldots, c_r(E, h^E)\right)$ weakly positive for any Griffiths positive vector bundle $(E, h^E)$ over a complex manifold $X$ of dimension $n, n \geqslant k$?
\end{question}
Recall that a real $(k,k)$-form $u$ is called weakly positive (resp. non-negative) if 
$u\wedge (\sqrt{-1})^{(n-k)^2}\beta\wedge \o{\beta}>0$ (resp. $\geq 0$) for any non-zero decomposable $(n-k,0)$-form $\beta=\beta_1\wedge \cdots\wedge\beta_{n-k}$, where $\beta_i$, $1\leq i\leq n-k$, are $(1,0)$-forms, see Definition \ref{various positivity} for the definitions of (weakly) positive (resp. non-negative) vector bundles. 

Griffiths \cite[Page 249]{Griffiths} proved that the second Chern form of a Griffiths positive vector bundle is positive by using Schwarz inequality. Guler \cite[Theorem 1.1]{MR2932990} verified Question \ref{Question} for all signed Segre forms, and Diverio-Fagioli \cite{DF} showed the positivity of several other polynomials in the Chern forms of a Griffiths (semi)positive vector bundle by considering the pushforward of a flag bundle, including the later developments \cite{Fag20, Fag22}. See Xiao \cite{MR4361967} and Ross-Toma \cite{RT} for other related results of ample vector bundles.

For Bott-Chern non-negative vector bundles, Bott-Chern \cite[Lemma 5.3, (5.5)]{MR185607} proved that all Chern forms are non-negative, Li \cite[Proposition 3.1]{MR4263677} extended Bott-Chern's result and obtained all Schur forms are non-negative. Later, Finski \cite[Theorem 2.15]{Fin} proved the equivalence of Bott-Chern non-negativity and dual Nakano non-negativity. Moreover, using a purely algebraic method, Finski \cite[Section 3.4]{Fin} proved that all Schur forms of a Nakano non-negative vector bundle are non-negative. For (dual) Nakano positive vector bundles, Finski \cite[Theorem 1.1]{Fin} proved that all Schur forms are positive by the refinement of the determinantal formula of Kempf-Laksov on the level of differential forms. However, as pointed out by Finski \cite[Remark 3.18]{Fin}, the above algebraic method can be used to deal with the case of non-negativity, while for the positivity statement, it is not clear if one can refine the algebraic method because there is no similar criterion for (dual) Nakano positivity (see \cite[Remark 2.16]{Fin}) and there is little known about the specific structure of the forms defined in \cite[(3.83)]{Fin}. This motivates the author to study the question of Griffiths (Question \ref{Question}) by developing the purely algebraic method. 

In \cite[Section 2.3]{Fin}, Finski introduced the definition of decomposably positive vector bundles, see Definition \ref{decom-pos},  which is a generalization of both Nakano positivity and dual Nakano positivity, and coincides with Griffiths positivity for $n\cdot r\leq 6$. So it is natural to wonder if Question \ref{Question} holds for such positive vector bundles. In this paper, we introduce two new notions of positivity of vector bundles, called strongly decomposable positivity of type I and type II, see Definition \ref{strongly} and Definition \ref{stongly-II}. They fall in between (dual) Nakano positivity and decomposable positivity. Roughly speaking, $(E,h^E)$ is strongly decomposably positive of type I if, for any $x\in X$, there is a decomposition $T_x^{1,0}X=U_x\oplus V_x$ such that it is Nakano positive in the subspaces  $E_x\otimes U_x$ and dual Nakano positive in the subspace $\o{E}_x\otimes V_x$, and the cross curvature terms vanish, see Definition \ref{strongly}. Using the purely algebraic method, we answer Question \ref{Question} affirmatively for strongly decomposably positive vector bundles of type I. 
\begin{theorem}\label{main-thm}
	Let $(E,h^E)$ be a strongly decomposably positive vector bundle of type I over a complex manifold $X$, $\mathrm{rank}E=r$, and $\dim X=n$. Then the Schur form 
	$P_\lambda(c(E,h^E))$ is weakly positive for any partition $\lambda\in \Lambda(k,r)$, $k\leq n$ and $k\in\mb{N}$.
\end{theorem}

The proof of the theorem above primarily involves presenting an equivalent characterization of a strongly decomposably positive vector bundle of type I. While previous algebraic methods can handle the non-negative cases, for the strictly positive situations, by utilizing our equivalent characterization, we can derive a contradiction if the Schur form is not weakly positive.

	From \cite[Theorem 1.2]{RA}, a real $(k,k)$-form $u$ is non-negative if and only if $u$ can be written as 
	\begin{equation}\label{non-negative}
  u=\sum_{s=1}^N (\sqrt{-1})^{k^2}\alpha_s\wedge \o{\alpha_s}
\end{equation}
 for some $(k,0)$-forms $\alpha_s,1\leq s\leq N$. By \eqref{psi} and \eqref{schur forms}, the Schur form has the following form
\begin{equation}\label{Schur-form0}
   P_\lambda(c(E,h^E))=\left(\frac{1}{2\pi}\right)^k\left(\frac{1}{k!}\right)^2(\sqrt{-1})^{k^2}\sum_{\rho,t,c,\epsilon}(-1)^{|\epsilon|+k}\psi_{\rho t c\epsilon}\wedge \o{\psi_{\rho t c\epsilon}}.
\end{equation}
where $\psi_{\rho t c\epsilon}$ is a $(|\epsilon|,k-|\epsilon|)$-form and is defined by
\begin{equation}\label{psi0}
  \psi_{\rho t c\epsilon}:=\sum_{\sigma\in S_k}q_{\sigma t}\bigwedge_{j=1}^k(\o{B_{\rho_{\sigma(j)}c_j}})^{\epsilon_j}\wedge (\o{A_{\rho_{\sigma(j)}c_j}})^{1-\epsilon_j}.
\end{equation}
It is not clear how to express \eqref{Schur-form0} as in \eqref{non-negative} in the general case. Hence, it seems hard to prove that the Schur form $P_\lambda(c(E,h^E))$ is a positive $(k,k)$-form by using the algebraic method. However, if $(E,h^E)$ is Nakano positive or dual Nakano positive, it is equivalent to $A=0$ or $B=0$. For example, for $A=0$, \eqref{psi0} gives 
\begin{equation*}
  \psi_{\rho t c \epsilon_1}=\sum_{\sigma\in S_k}q_{\sigma t}\bigwedge_{j=1}^k\o{B_{\rho_{\sigma(j)}c_j}},\quad \epsilon_1=(1,\cdots,1)
\end{equation*}
and $\psi_{\rho t c \epsilon}=0\text{ for any }\epsilon\neq \epsilon_1$. Then the Schur form is given by
\begin{equation*}
  P_\lambda(c(E,h^E))=\left(\frac{1}{2\pi}\right)^k\left(\frac{1}{k!}\right)^2(\sqrt{-1})^{k^2}\sum_{\rho,t,c}\psi_{\rho t c \epsilon_1}\wedge \o{\psi_{\rho t c\epsilon_1}},
\end{equation*}
 which satisfies \eqref{non-negative} because $ \psi_{\rho t c \epsilon_1}$ is a $(k,0)$-form.
As a result, we can give an algebraic proof of the following positivity of Schur forms for (dual) Nakano positive vector bundles. 
\begin{theorem}[{Finski \cite[Theorem 1.1]{Fin}}]\label{main-cor}
	Let $(E, h^E)$  be a (dual) Nakano positive vector bundle of rank $r$ over a complex manifold $X$ of dimension $n$. Then for any $k \in \mathbb{N}$, $k \leqslant n$, and $\lambda \in \Lambda(k, r)$, the $(k, k)$-form $P_\lambda\left(c\left(E, h^E\right)\right)$ is positive.
\end{theorem}

Inspired by the definition of the strongly decomposable positivity of type I,
it is natural to define the strongly decomposable positivity of type II by decomposing the vector bundle, which is the direct sum of Nakano positive and dual Nakano positive vector bundles point-wisely, see Definition \ref{stongly-II}. By Littlewood-Richardson rule (see \cite[Chapter 5]{MR1464693}), the Schur class of direct sum $E\oplus F$ can be given by
\begin{equation*}
  P_\lambda(c(E\oplus F))=\sum_{\mu,\nu}c^\lambda_{\mu\nu}P_\mu(c(E))P_\nu(c(F)),
\end{equation*}
where $c^\lambda_{\mu\nu}(\geq0)$ is a Littlewood-Richardson coefficient.
In this paper, by refining the above identity on the level of differential forms and using Theorem \ref{main-cor}, we obtain
\begin{theorem}\label{main theorem2}
	Let $(E,h^E)$ be a strongly decomposably positive vector bundle of type II over a complex manifold $X$, $\mathrm{rank} E=r$, and $\dim X=n$. Then the Schur form 
	$P_\lambda(c(E,h^E))$ is positive for any partition $\lambda\in \Lambda(k,r)$, $k\leq n$ and $k\in\mb{N}$.
\end{theorem}
\begin{remark}
Comparing Theorem \ref{main-thm} with Theorem \ref{main theorem2}, it is natural to ask if the Schur forms are positive for a strongly decomposably vector bundle of type I.	
\end{remark}

The article is organized as follows: In Section \ref{sec:strongly}, we will define two types of strongly decomposably positive vector bundles, which are the generalizations of both Nakano positivity and dual Nakano positivity, and are stronger than decomposable positivity.
In Section \ref{sec:positivity}, we will recall the positivity notions for differential forms and show the positivity of the product of two positive forms. In Section \ref{sec:pos}, we will give a criterion of a strongly decomposably positive vector bundle of type I, recall the definitions of Schur forms and Griffiths cone, then prove the weak positivity of Schur forms, Theorem \ref{main-thm} and Theorem \ref{main-cor} are established in this section. In Section \ref{strongly positive II}, we will give a criterion of a strongly decomposably positive vector bundle of type II and prove the positivity of Schur forms, Theorem \ref{main theorem2} is established in this section.

\vspace{5mm}
\noindent{\it Acknowledgements.} The author thanks Siarhei Finski for helpful discussions and anonymous referees for valuable comments that improved our article.

\section{Strongly decomposably positive vector bundles}\label{sec:strongly}

This section will define two types of strongly decomposably positive vector bundles.

\subsection{Connections and curvatures}

In this subsection, we will recall the definitions of the Chern connection and its curvature for a Hermitian holomorphic vector bundle. One can refer to \cite[Chapter 1]{Kobayashi+1987} for more details. We will use the Einstein summation convention in this paper. 

Let $\pi:(E,h^E)\to X$ be a Hermitian holomorphic vector bundle over a complex manifold $X$, $\mr{rank} E=r$ and $\dim X=n$. Let $\n^E$ be the Chern connection of $(E,h^E)$, which preserves the metric $h^E$ and is of $(1,0)$-type. With respect to a local holomorphic frame $\{e_i\}_{1\leq i\leq r}$ of $E$, one has
\begin{equation}\label{Chern-connection}
  \n^E e_i=\theta^j_i e_j,
\end{equation}
where $\theta=(\theta^j_i)$ ($j$ row, $i$ column) is the connection form of $\n^E$.
 More precisely, 
\begin{equation*}
  \theta^j_i=\p h_{i\b{k}}h^{\b{k}j},
\end{equation*}
where $h_{i\b{k}}:=h(e_i,{e_k})$. In terms of matrix form, it is 
\begin{equation*}
  \theta^\top=\p h\cdot h^{-1}.
\end{equation*}
 Considering $e=(e_1,\cdots,e_r)$ as a row vector, then \eqref{Chern-connection} can be written as 
\begin{equation*}
  \n^E e=e\cdot\theta.
\end{equation*}
Let $R^E=(\n^E)^2\in A^{1,1}(X,\mr{End}(E))$ denote the Chern curvature of $(E,h^E)$, and write 
\begin{equation*}
  R^E=R^j_ie_j\otimes e^i \in A^{1,1}(X,\mr{End}(E)),
\end{equation*}
where $R=(R^j_i)$ ($j$ row, $i$ column) is the curvature matrix whose entries are $(1,1)$-forms,  $\{e^i\}_{1\leq i\leq r}$ denotes the dual frame of $\{e_i\}_{1\leq i\leq r}$. The curvature matrix $R=(R^j_i)$ is given by
\begin{equation*}
  R^{j}_i=d\theta^j_i+\theta^j_k\wedge \theta^k_i=\b{\p}\theta^j_i.
\end{equation*}
 If $\{\tilde{e}_i\}_{1\leq i\leq r}$ is another local holomorphic frame of $E$ with 
$\tilde{e}_i=a_i^je_j$, then 
\begin{equation*}
  \tilde{e}=e\cdot a,
\end{equation*}
where $\tilde{e}:=(\tilde{e}_1,\cdots,\tilde{e}_r)$ and $a=(a^i_j)$. Denote by $\wt{R}$ the curvature matrix  with respect to the local frame $\{\tilde{e}_i\}_{1\leq i\leq r}$. Then
\begin{equation}\label{Tran}
  \wt{R}=a^{-1}\cdot R\cdot a.
\end{equation}
Let $\{z^\alpha\}_{1\leq\alpha\leq n}$ be local holomorphic coordinates of $X$. Write  
\begin{equation*}
  R^j_i=R^j_{i\alpha\b{\beta}}dz^\alpha\wedge d\b{z}^\beta
\end{equation*}
and denote
\begin{equation*}
  R_{i\b{j}}:=R^{k}_{i}h_{k\b{j}}=R_{i\b{j}\alpha\b{\beta}}dz^\alpha\wedge d\b{z}^\beta,
\end{equation*}
so that
\begin{equation*}
  R_{i\b{j}\alpha\b{\beta}}=R^k_{i\alpha\b{\beta}}h_{k\b{j}}=-\p_\alpha\p_{\b{\beta}}h_{i\b{j}}+h^{\b{l}k}\p_\alpha h_{i\b{l}}\p_{\b{\beta}}h_{k\b{j}},
\end{equation*}
where $\p_\alpha:=\p/\p z^\alpha$ and $\p_{\b{\beta}}:=\p/\p\b{z}^\beta$.

\subsection{Strongly decomposable positivity}

This subsection will define two types of strongly decomposably positive vector bundles.
Firstly, we recall the following definitions of Nakano positive and dual Nakano positive vector bundles. 
 \begin{definition}[(Dual) Nakano positive] 
 	A Hermitian holomorphic vector bundle $(E,h^E)$ is called Nakano positive (resp. non-negative) if 
 	\begin{equation*}
  R_{i\b{j}\alpha\b{\beta}}u^{i\alpha}\o{u^{j\beta}}>0 \,(\text{resp. } \geq 0 )
\end{equation*}
 for any non-zero element $u=u^{i\alpha}e_i\otimes \p_\alpha\in E\otimes T^{1,0}X$. 
 $(E,h^E)$ is called dual Nakano positive (non-negative) if 
 \begin{equation*}
  R_{i\b{j}\alpha\b{\beta}}v^{\b{j}\alpha}\o{v^{\b{i}\beta}}>0 \,(\text{resp. } \geq 0 )
\end{equation*}
for any non-zero element $v=v^{\b{j}\alpha}\b{e}_j\otimes \p_\alpha\in \o{E}\otimes T^{1,0}X$.
 \end{definition}
In \cite[Definition 2.18]{Fin}, S. Finski introduced the following new notion of positivity for vector bundles: decomposable positivity. 
\begin{definition}[Decomposably positive]\label{decom-pos}
	A Hermitian vector bundle $\left(E, h^E\right)$ is called decomposably non-negative if for any $x \in X$, there is a number $N \in \mathbb{N}$ and linear (respectively sesquilinear) forms $l_p^{\prime}: T_x^{1,0} X \otimes E_x \rightarrow \mathbb{C}$ (respectively $l_p: T_x^{1,0} X \otimes E_x \rightarrow \mathbb{C}$), $p=1, \ldots, N$, such that for any $v \in T_x^{1,0} X, \xi \in E_x$, we have
$$
\frac{1}{2 \pi}\left\langle R_x^E(v, \bar{v}) \xi, \xi\right\rangle_{h^E}=\sum_{p=1}^N\left|l_p(v, \xi)\right|^2+\sum_{p=1}^N\left|l_p^{\prime}(v, \xi)\right|^2 .
$$
We say that it is decomposably positive if, moreover, $\left\langle R_x^E(v, \bar{v}) \xi, \xi\right\rangle_{h^E} \neq 0$ for $v, \xi \neq 0$.
\end{definition}
 \begin{remark}
From the above definition, a (dual) Nakano positive vector bundle must be decomposably positive. From \cite[Proposition 2.21]{Fin}, for $n\cdot r\leq 6$ decomposable positivity is equivalent to Griffiths positivity, i.e. 
\begin{equation}\label{Griffiths-positive}
  R_{i\b{j}\alpha\b{\beta}}v^i\b{v}^j\xi^\alpha\b{\xi}^\beta>0
\end{equation}
 for any non-zero $\xi=\xi^\alpha \p_\alpha\in T^{1,0}X$ and $v=v^ie_i\in E$. Decomposable positivity is strictly stronger than Griffiths positivity for all other $n, r \neq 1$. 
 \end{remark}
Next, we will introduce the notion of strongly decomposable positivity, which falls in between (dual) Nakano positivity and decomposable positivity.
\begin{definition}[Strongly decomposably positive of type I]\label{strongly} 
A Hermitian vector bundle $(E,h^E)$ is called a strongly decomposably positive (resp. non-negative) vector bundle of type I if, for any $x\in X$, there exists a decomposition $T^{1,0}_x X=U_x\oplus V_x$ such that 
\begin{equation*}
 R_{i\b{j}\alpha\b{\beta}}u^{i\alpha}\o{u^{j\beta}}>0(\text{resp. }\geq 0),  \,\, R_{i\b{j}\alpha\b{\beta}}u^{i\alpha}\o{v'^{j\beta}}=0,\,\, R_{i\b{j}\alpha\b{\beta}}v^{\b{j}\alpha}\o{v^{\b{i}\beta}}>0(\text{resp. }\geq 0)
\end{equation*}
 for any non-zero elements $u=u^{i\alpha}e_i\otimes \p_\alpha\in E_x\otimes U_x$, $v'=v'^{j\beta}e_j\otimes \p_\beta\in E_x\otimes V_x$ and $v=v^{\b{j}\alpha}\b{e}_j\otimes \p_\alpha\in \o{E}_x\otimes V_x$.
 \end{definition}
 \begin{remark}
 	In fact, by taking $u^{i\alpha}=1=v'^{j\beta}$ in the above definition, the condition $R_{i\b{j}\alpha\b{\beta}}u^{i\alpha}\o{v'^{j\beta}}=0$ is equivalent to $R_{i\b{j}\alpha\b{\beta}}=0$ for any $\p_\alpha\in U_x$ and $\p_\beta\in V_x$, i.e. $R(U_x,\o{V_x})=0$, which is also equivalent to $R_{i\b{j}\alpha\b{\beta}}v^{\b{j}\alpha}\o{u'^{\b{i}\beta}}=0$ for any $v=v^{\b{j}\alpha}\b{e}_j\otimes \p_\alpha\in \o{E}_x\otimes V_x$ and $u'=u'^{\b{i}\beta}\b{e}_i\otimes\p_\beta\in \o{E}_x\otimes U_x$. Hence, the above definition is invariant under switching $U$ and $V$. 
 \end{remark}

For any point $x\in X$, if $V_x=\{0\}$ (resp. $U_x=\{0\}$), then strongly decomposable positivity of type I is exactly Nakano positive (resp. dual Nakano positive). 
\begin{example}
Let $\pi_1:(E_1,h^{E_1})\to X_1$ be a Nakano positive vector bundle and 	$\pi_2:(E_2,h^{E_2})\to X_2$ be a dual Nakano positive vector bundle. Denote by $p_i:X_1\times X_2\to X_i$, $i=1,2$, the natural projections, then 
\begin{equation*}
  \pi_{\oplus}:(p_1^*E_1\oplus p_2^*E_2,p_1^*h^{E_1}\oplus p_2^*h^{E_2})\to X_1\times X_2
\end{equation*}
is a strongly decomposably non-negative vector bundle of type I, and 
\begin{equation*}
  \pi_{\otimes}:(p_1^*E_1\otimes p_2^*E_2,p_1^*h^{E_1}\otimes p_2^*h^{E_2})\to X_1\times X_2
\end{equation*}
is a strongly decomposably positive vector bundle of type I. 

\end{example}

From Definition \ref{strongly}, a strongly decomposably positive vector bundle  $(E,h^E)$ of type I means that there is a decomposition of holomorphic tangent bundle $T^{1,0}_x X=U_x\oplus V_x$, such that $(E,h^E)$ is Nakano positive in $E_x\otimes U_x$, dual Nakano positive in $\o{E}_x\otimes V_x$, and the cross curvature terms vanish. Naturally, one may define another strongly decomposable positivity of vector bundles by decomposing the vector bundle. More precisely, 
\begin{definition}[Strongly decomposably positive of type II]\label{stongly-II}
We call $(E,h^E)$ a strongly decomposably positive (resp. non-negative) vector bundle of type II if, for any $x\in X$, there is an orthogonal decomposition of $(E_x,h^{E}|_{E_x})$, $E_x=E_{1,x}\oplus E_{2,x}$, such that the Chern curvature $R^E_x$ has the following form
 \begin{equation*}
  R^E_x=\begin{pmatrix}
  R^{E}_x|_{E_{1,x}}&0 \\
  0& R^{E}_x|_{E_{2,x}}
\end{pmatrix},
\end{equation*}
and $R_{i\b{j}\alpha\b{\beta}}u^{i\alpha}\o{u^{j\beta}}>0$ (resp. $\geq 0$) for any non-zero $u=u^{i\alpha} e_i\otimes \p_\alpha\in E_{1,x}\otimes T^{1,0}_xX$, $R_{i\b{j}\alpha\b{\beta}}v^{i\b{\beta}}\o{v^{j\b{\alpha}}}>0$ (resp. $\geq 0$) for any non-zero $v=v^{i\b{\beta}}e_i\otimes \p_{\b{\beta}}\in E_{2,x}\otimes T^{0,1}_xX$.
\end{definition}
A simple example of a strongly decomposably vector bundle of type II is as follows. 
\begin{example}
	Let $(E,h^E)$ be a Nakano positive vector bundle and $(F,h^F)$ be a dual Nakano positive vector bundle over a complex manifold $X$. Then $(E\oplus F, h^E\oplus h^F)$ is a strongly decomposably positive vector bundle of type II. 
\end{example}
By Definition \ref{strongly} and Definition \ref{stongly-II}, a strongly decomposably positive vector bundle is defined as follows.
\begin{definition}[Strongly decomposably positive]\label{Defn-strongly}
A Hermitian vector bundle $(E,h^E)$ is called strongly decomposably positive if it is a strongly decomposably positive vector bundle of type I or type II.
\end{definition}

Similarly, one can define strongly decomposably negative (non-positive) vector bundles. Note that the dual of the Nakano positive (negative) vector bundle is dual Nakano negative (positive), so 
\begin{proposition}
	A Hermitian vector bundle $(E,h^E)$ is a strongly decomposably positive (non-negative) vector bundle of type I (type II) if and only if $(E^*,h^{E^*})$ is a strongly decomposably negative (non-positive) vector bundle of type I (type II). 
\end{proposition}
\begin{remark}
Let $(E,h^E)$ be a strongly decomposably positive vector bundle, and $Q$ be a quotient bundle of $E$. The curvature of the bundle $Q$ is given by 
\begin{equation*}
  R^Q=\left.R^E\right|_Q+C \wedge \o{C}^\top
\end{equation*}
for some matrix $C$ whose entries are $(1,0)$-forms. From the criteria of strongly decomposably positive vector bundles, Theorem \ref{criterion} and Theorem \ref{prop-criterion}, and the above curvature formula of quotient bundles, the quotient bundle $(Q,h^Q)$ cease to be a strongly decomposably positive vector bundle in general. 
\end{remark}

\subsection{Relation to decomposable positivity}

From the equivalent descriptions of Nakano non-negative and dual non-negative due to S. Finski \cite[Theorem 2.15 and 2.17]{Fin}, one has 
\begin{proposition}\label{prop1}
A Hermitian vector bundle $(E,h^E)$ is decomposably non-negative if and only if the Chern curvature matrix has the following form 
\begin{equation}\label{decom-pos0}
  R=-B\wedge\o{B}^\top+A\wedge\o{A}^\top
\end{equation}
 	with respect to a unitary frame, where $A$ (resp. $B$) is a $r\times N$ matrix with $(1,0)$-forms (resp. $(0,1)$-forms) as entries.
\end{proposition}
\begin{proof}
 	 From Definition \ref{decom-pos}, $(E,h^E)$ is decomposably positive if for any $x \in X$, there is a number $N \in \mathbb{N}$ and linear (respectively sesquilinear) forms $l_p^{\prime}: T_x^{1,0} X \otimes E_x \rightarrow \mathbb{C}$ (respectively $l_p: T_x^{1,0} X \otimes E_x \rightarrow \mathbb{C}$), $p=1, \ldots, N$, such that for any $v \in T_x^{1,0} X, \xi \in E_x$, we have
 	 \begin{equation}\label{decom-pos1}
  \left\langle R_x^E(v, \bar{v}) \xi, \xi\right\rangle_{h^E}=\sum_{p=1}^N\left|l_p(v, \xi)\right|^2+\sum_{p=1}^N\left|l_p^{\prime}(v, \xi)\right|^2.
\end{equation}
We denote 
\begin{equation*}
  l_p(v,\xi)=l_{ip\b{\beta}}v^i\b{\xi}^\beta,\quad l'_p(v,\xi)=l'_{ip\alpha}v^i\xi^\alpha,
\end{equation*}
and set $A=(A_{jp})$ and $B=(B_{jp})$ by 
\begin{equation*}
  A_{jp}:=A_{jp\alpha}dz^\alpha=\o{l_{jp\b{\alpha}}}dz^\alpha,\quad B_{jp}:=B_{jp\b{\beta}}d\b{z}^\beta=\o{l'_{jp\beta}}d\b{z}^\beta.
\end{equation*}
Then \eqref{decom-pos1} is equivalent to
\begin{align}\label{decom-pos2}
\begin{split}
  R_{i\b{j}\alpha\b{\beta}}&=\sum_{p=1}^N l_{ip\b{\beta}}\o{l_{jp\b{\alpha}}}+\sum_{p=1}^N l'_{ip\alpha}\o{l'_{jp\beta}}\\
  &=\sum_{p=1}^N A_{jp\alpha}\o{A_{ip\beta}}+\sum_{i=1}^N B_{jp\b{\beta}}\o{B_{ip\b{\alpha}}}.
 \end{split}
\end{align}
With respect to a unitary frame, $R^j_{i\alpha\b{\beta}}=R_{i\b{j}\alpha\b{\beta}}$ and \eqref{decom-pos2} is equivalent to 
\begin{equation*}
  R=-B\wedge\o{B}^\top+A\wedge\o{A}^\top,
\end{equation*}
which completes the proof. 
\end{proof}
From Proposition \ref{prop1} and Definition \ref{decom-pos}, $(E,h^E)$ is decomposably positive if and only if \eqref{decom-pos0} holds and $\left\langle R_x^E(v, \bar{v}) \xi, \xi\right\rangle_{h^E} \neq 0$ for $v, \xi \neq 0$. By Theorem \ref{criterion} and Theorem \ref{prop-criterion}, we have
\begin{corollary}
	If $(E,h^E)$ is a strongly decomposably positive vector bundle of type I or type II, then $(E,h^E)$ is decomposably positive. 
\end{corollary}
On the other hand, from Theorem \ref{criterion} and Theorem \ref{prop-criterion}, the two types of strongly decomposably positive vector bundles can not contain each other. Both are the generalizations of (dual) Nakano positive vector bundles and are stronger than decomposable positivity. One can refer to the following Figure 1. 
\begin{figure}[ht]
\centering
\includegraphics[width=0.7\textwidth]{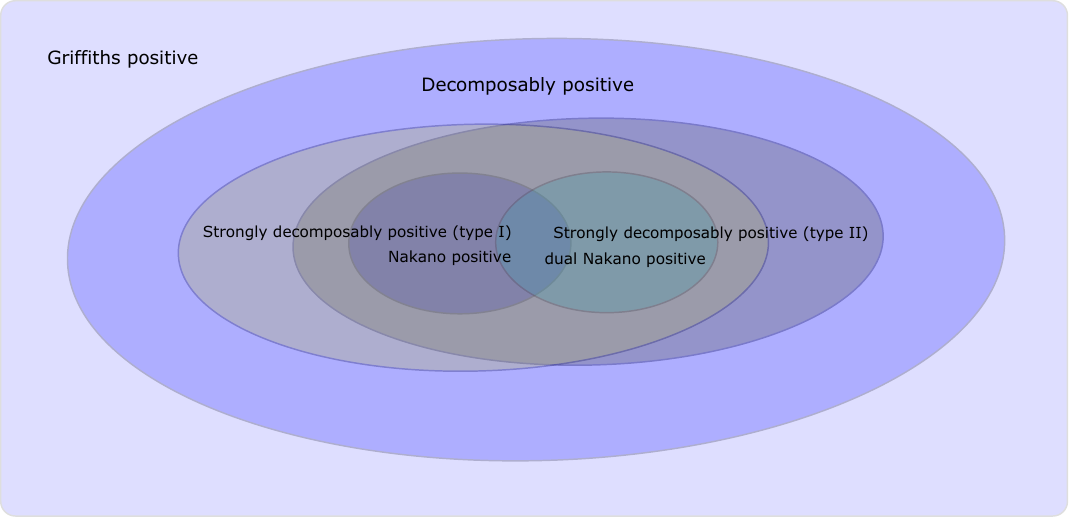}
\caption{Relations of several notions of positivity}
\end{figure}
\begin{remark}
Note that the set of curvature operators of a vector bundle, whether they are Griffiths positive, decomposably positive, or (dual) Nakano positive, is closed under addition. Specifically, if both \(R_1\) and \(R_2\) are curvature operators that fall into any of these categories, then their sum \(R_1 + R_2\) will also belong to the same category. Moreover, a decomposably positive curvature operator can be expressed as the sum of a Nakano positive curvature operator and a dual Nakano positive curvature operator. As a result, it is also the sum of the strongly decomposable positivity of type I (or type II). Since the strongly decomposable positivity of type I (type II) is strictly stronger than decomposable positivity, the set of strongly decomposable positive curvature operators of type I (or type II) is not closed under addition. 
\end{remark}

\section{Positivity notions for differential forms}\label{sec:positivity}

In this section, we will recall positivity notions for differential forms. For more details, one can refer to \cite[Section 1.1]{Fag20} and \cite{RA, Fin}. 

Let $V$ be a complex vector space of dimension $n$ and let $(e_1,\cdots,e_n)$ be a basis of $V$. Denote by $(e^1,\cdots,e^n)$ the dual basis of $V^*$. Let $\Lambda^{p,q}V^*$ denote the space of $(p,q)$-forms, and $\Lambda^{p,p}_{\mb{R}}V^*\subset \Lambda^{p,p}V^*$ be the subspace of real $(p,p)$-forms. 
\begin{definition}\label{various positivity}
A form $\nu\in \Lambda^{n,n}V^*$ is called a non-negative (resp. positive) volume form if $\nu=\tau\sqrt{-1}e^1\wedge \o{e}^1\wedge \cdots\wedge \sqrt{-1}e^n\wedge \o{e}^n$ for some $\tau\in \mb{R}$, $\tau \geq 0$ (resp. $\tau>0$).
\end{definition}
Now we set $q=n-p$, a $(q,0)$-form $\beta$ is called decomposable if $\beta=\beta_1\wedge \cdots\wedge \beta_q$ for some $\beta_1,\ldots,\beta_q\in V^*$. 
\begin{definition}
	A real $(p,p)$-form $u\in \Lambda^{p,p}_{\mb{R}}V^*$ is called 
	\begin{itemize}
  \item weakly non-negative (resp. weakly positive), if for every non-zero $\beta\in \Lambda^{q,0}V^*$ decomposable, $u\wedge (\sqrt{-1})^{q^2}\beta\wedge\b{\beta}$ is a non-negative (resp. positive) volume form;
  \item non-negative (resp. positive), if for every non-zero $\beta\in \Lambda^{q,0}V^*$, $u\wedge (\sqrt{-1})^{q^2}\beta\wedge\b{\beta}$ is a non-negative (resp. positive) volume form;
  \item strongly non-negative (resp. strongly positive) if there are decomposable forms $\alpha_1,\ldots,\alpha_N\in \Lambda^{p,0}V^*$ such that $u=\sum_{s=1}^N (\sqrt{-1})^{p^2}\alpha_s\wedge \o{\alpha_s}$.
\end{itemize}
\end{definition}
\begin{remark}
Let $\mathrm{WP}^p V^{*}, \mathrm{P}^p V^{*}$ and $\mathrm{SP}^p V^{*}$ denote respectively the closed positive convex cones contained in $\Lambda_{\mathbb{R}}^{p, p} V^{\vee}$ spanned by weakly non-negative, non-negative and strongly non-negative forms. Then
  \begin{equation}
  \mathrm{SP}^p V^{*} \subseteq \mathrm{P}^p V^{*} \subseteq \mathrm{WP}^p V^{*}.
\end{equation}
Note that the above two inclusions become equalities for $p=0,1,n-1,n$, and the inclusions are strict for $2\leq p\leq n-2$, see e.g. \cite[Remark 1.7, 1.8]{Fag20} and \cite{RA}. 
\end{remark}

\begin{proposition}\label{product}
	If $u$ is a positive $(k,k)$-form and $v$ is a positive $(l,l)$-form,  $k+l\leq n$, then $u\wedge v$ is a positive $(k+l,k+l)$-form. 
\end{proposition}
\begin{proof}
By \cite[Corollary 1.3 (a)]{RA}, $u\wedge v$ is a non-negative $(k+l,k+l)$-form, i.e. for any non-zero $\beta\in \Lambda^{n-k-l,0}V^*$, 
\begin{equation}\label{prod-positive}
  u\wedge v\wedge (\sqrt{-1})^{(n-k-l)^2}\beta\wedge\b{\beta}\geq 0.
\end{equation}
By \cite[Theorem 1.2]{RA}, $v$ has the following form
\begin{equation*}
  v=\sum_{s=1}^N (\sqrt{-1})^{l^2}\alpha_s\wedge \o{\alpha_s}
\end{equation*}
for some $(l,0)$-forms $\alpha_s$, $1\leq s\leq N$. So
\begin{align*}
\begin{split}
  u\wedge v\wedge (\sqrt{-1})^{(n-k-l)^2}\beta\wedge\b{\beta}&=\sum_{s=1}^N u\wedge (\sqrt{-1})^{l^2}\alpha_s\wedge \o{\alpha_s}\wedge (\sqrt{-1})^{(n-k-l)^2}\beta\wedge\b{\beta}\\
  &=\sum_{s=1}^N u\wedge (\sqrt{-1})^{(n-k)^2}\alpha_s\wedge\beta\wedge \o{\alpha_s\wedge\beta}.
 \end{split}
\end{align*}
Thus the equality in \eqref{prod-positive} holds if and only if 
\begin{equation*}
  u\wedge (\sqrt{-1})^{(n-k)^2}\alpha_s\wedge\beta\wedge \o{\alpha_s\wedge\beta}=0,\quad 1\leq s\leq N,
\end{equation*}
which is equivalent to 
\begin{equation*}
  \alpha_s\wedge \beta=0, \quad 1\leq s\leq N.
\end{equation*}
Thus
\begin{align*}
\begin{split}
  v\wedge (\sqrt{-1})^{(n-k-l)^2}\beta\wedge\b{\beta}=\sum_{s=1}^N  (\sqrt{-1})^{(n-k)^2}\alpha_s\wedge\beta\wedge \o{\alpha_s\wedge\beta}=0,
 \end{split}
\end{align*}
which contradicts the positivity of $v$. Hence 
\begin{equation*}
   u\wedge v\wedge (\sqrt{-1})^{(n-k-l)^2}\beta\wedge\b{\beta}>0
\end{equation*}
for any non-zero $\beta\in \Lambda^{n-k-l,0}V^*$, i.e. $u\wedge v$ is positive.

\end{proof}

Let $X$ be a complex manifold of dimension $n$ and denote by $A^{p,q}(X)$ the space of all smooth $(p,q)$-forms. 
\begin{definition}
	A real $(p,p)$-form $\alpha\in A^{p,p}(X)$ is called weakly non-negative (weakly positive), non-negative (positive), or strongly non-negative (strongly positive) if for any $x\in X$, $\alpha_x\in \Lambda^{p,p}_{\mb{R}}(T_x^{1,0}X)^*$ is weakly non-negative (weakly positive), non-negative (positive), or strongly non-negative (strongly positive) respectively. 
\end{definition}

\section{Strongly decomposable positivity of type I}\label{sec:pos}

In this section, we will give a criterion of strongly decomposably positive vector bundles of type I and prove the weak positivity of Schur forms. 

\subsection{A criterion of type I positivity}

In this subsection, following S. Finski's approach \cite[Theorem 2.15, 2.17]{Fin}, we will give a criterion for the strongly decomposable positivity (non-negativity) of type I by using M.-D. Choi's results. 

Let $(E,h^E)$ be a strongly decomposably non-negative vector bundle of type I. For any $x\in X$, there exists a decomposition $T^{1,0}_xX=U_x\oplus V_x$. One can take local holomorphic coordinates $\{z^1,\cdots, z^n\}$ around $x$ such that 
\begin{equation*}
  U_x=\mathrm{span}_{\mb{C}}\{\p_1,\cdots, \p_{n_0}\},\,\, V_x=\mathrm{span}_{\mb{C}}\{\p_{n_0+1},\cdots, \p_{n}\},
\end{equation*}
where $n_0:=\dim U_x$ and recall that $\p_\alpha:=\p/\p z^\alpha$. Let $\{e_i\}_{1\leq i\leq r}$ be a local holomorphic frame of $E$ such that 
\begin{equation*}
  h_{i\b{j}}(x)=h^E(e_i(x),e_j(x))=\delta_{ij}.
\end{equation*}
With respect to $\{z^\alpha\}_{1\leq \alpha\leq n}$ and $\{e_i\}_{1\leq i\leq r}$, the Chern curvature matrix $R=(R^j_i)$ at $x\in X$ has the following expression
\begin{align}\label{cur1}
\begin{split}
  R^j_i&=R^j_{i\alpha\b{\beta}}dz^\alpha\wedge d\b{z}^\beta\\
  &=\sum_{\alpha,\beta=1}^{n_0}R_{i\alpha\b{\beta}}^j dz^\alpha\wedge d\b{z}^\beta+\sum_{\alpha,\beta=n_0+1}^{n}R_{i\alpha\b{\beta}}^j dz^\alpha\wedge d\b{z}^\beta\\
  &\quad+\sum_{\alpha=1}^{n_0}\sum_{\beta=n_0+1}^{n}R_{i\alpha\b{\beta}}^j dz^\alpha\wedge d\b{z}^\beta+\sum_{\alpha=n_0+1}^{n}\sum_{\beta=1}^{n_0}R_{i\alpha\b{\beta}}^j dz^\alpha\wedge d\b{z}^\beta.
 \end{split}
\end{align}
By assumption, $(E,h^E)$ is strongly decomposably non-negative of type I, so 
\begin{align*}
\sum_{\alpha=1}^{n_0}\sum_{\beta=n_0+1}^nR_{i\b{j}\alpha\b{\beta}}u^{i\alpha}\o{v'^{j\beta}}=0	
\end{align*}
for any $\sum_{\alpha=1}^{n_0}u^{i\alpha}e_i\otimes \p_\alpha$ and $\sum_{\beta={n_0+1}}^nv'^{j\beta}e_j\otimes \p_\beta$, which follows that 
\begin{equation}\label{cur2}
  R_{i\b{j}\alpha\b{\beta}}=0,\quad 1\leq \alpha\leq n_0,\,n_0+1\leq \beta\leq n.
\end{equation}
By conjugation, one gets
\begin{equation}\label{cur3}
  R_{i\b{j}\alpha\b{\beta}}=\o{R_{j\b{i}\beta\b{\alpha}}}=0,\quad n_0+1\leq \alpha\leq n,\,1\leq \beta\leq n_0.
\end{equation}
Substituting \eqref{cur2}, \eqref{cur3} into \eqref{cur1}, one has
\begin{equation*}
  R^j_i=\sum_{\alpha,\beta=1}^{n_0}R_{i\alpha\b{\beta}}^j dz^\alpha\wedge d\b{z}^\beta+\sum_{\alpha,\beta=n_0+1}^{n}R_{i\alpha\b{\beta}}^j dz^\alpha\wedge d\b{z}^\beta.
\end{equation*}

For the local  frame $\{\p_\alpha\}_{1\leq \alpha\leq n}$, we define a local metric $g$ around $x$ by 
\begin{equation*}
 g_{\alpha\b{\beta}}= g(\p_\alpha,\p_\beta):=\delta_{\alpha\beta}.
\end{equation*}
Now we define a linear map 
\begin{equation*}
  H^V_x: \mr{End}(V_x)\to \mr{End}(E_x)
\end{equation*}
by 
\begin{equation}\label{H-op}
  H_x^V(\p_\alpha\otimes dz^\gamma)=R^j_{i  \alpha \bar{\beta}} g^{\bar{\beta} \gamma}  e_j \otimes e^i=R^j_{i  \alpha \bar{\gamma}}   e_j \otimes e^i
\end{equation}
for any $n_0+1\leq \alpha,\gamma\leq n$. With respect to the basis $\{\p_\alpha\}_{n_0+1\leq \alpha \leq n}$, the matrix of $\p_\alpha\otimes dz^\gamma\in \mr{End}(U_x)$ is $E_{\alpha\gamma}$, which is the $(n-n_0)\times (n-n_0)$ matrix with $1$ at the $(\alpha,\gamma)$-component and zeros elsewhere. The matrix of $R^j_{i  \alpha \bar{\gamma}}  e_j \otimes e^i\in \mr{End}(E_x)$ is given by 
$(R^j_{i  \alpha \bar{\gamma}})_{1\leq j,i\leq r}$ ($j$ row, $i$ column). In terms of matrices, \eqref{H-op} becomes
\begin{equation}\label{H-1}
  H^V_x(E_{\alpha\gamma})=(R^j_{i  \alpha \bar{\gamma}})_{1\leq j,i\leq r}=(R_{i \b{j} \alpha \bar{\gamma}})_{1\leq j,i\leq r}.
\end{equation}
Then $(H^V_x(E_{\alpha\gamma}))_{n_0+1\leq \alpha,\gamma\leq n}$ is a $(n-n_0)\times (n-n_0)$ block matrix with $r\times r$ matrices as entries, and 
\begin{align*}
\begin{split}
&(H^V_x(E_{\alpha\gamma}))_{n_0+1\leq \alpha,\gamma\leq n}=\left((R_{i \b{j} \alpha \bar{\gamma}})_{1\leq j,i\leq r}\right)_{n_0+1\leq \alpha,\gamma\leq n} \quad (j\, \alpha\text{ row}, i\,\gamma\text{ column}).
 \end{split}
\end{align*}
Since $(E,h^E)$ is strongly decomposably non-negative of type I, so 
\begin{equation*}
   R_{i\b{j}\alpha\b{\beta}}v^{\b{j}\alpha}\o{v^{\b{i}\beta}}\geq 0
\end{equation*}
 for any non-zero $v=v^{\b{j}\alpha}e_i\otimes \p_\alpha\in E_x\otimes V_x$,
which follows that the matrix 
$$(H^V_x(E_{\alpha\gamma}))_{n_0+1\leq \alpha,\gamma\leq n}$$ is positive semi-definite. By using \cite[Theorem 2 and Theorem 1]{Choi}, there exist $(n-n_0)\times r$ matrices $V_p, 1\leq p\leq N_1$ (one can choose $N_1=(n-n_0)\cdot r$) such that 
\begin{equation*}
  H^V_x(E_{\alpha\gamma})=\sum_{p=1}^{N_1} \o{V_p}^\top\cdot E_{\alpha\gamma}\cdot V_p
\end{equation*}
for any $n_0+1\leq\alpha,\gamma\leq n$.
Combining with \eqref{H-1} and considering the $(j,i)$ entry, one has
\begin{equation*}
  R^j_{i\alpha\b{\beta}}=\sum_{p=1}^{N_1} (\o{V_p}^\top\cdot E_{\alpha\beta}\cdot V_p)_{j,i}=\sum_{p=1}^{N_1}\o{(V_p)_{\alpha j}}(V_p)_{\beta i}.
\end{equation*}
Hence 
\begin{align*}
\begin{split}
  \sum_{\alpha,\beta=n_0+1}^{n}R^j_{i\alpha\b{\beta}}dz^\alpha\wedge d\b{z}^\beta &=\sum_{\alpha,\beta=n_0+1}^{n}\sum_{p=1}^{N_1}\o{(V_p)_{\alpha j}}(V_p)_{\beta i}dz^\alpha\wedge d\b{z}^\beta\\
  &=\sum_{p=1}^{N_1}A_{jp}\wedge \o{A_{i p}},
 \end{split}
\end{align*}
where $A_{jp}:=\sum_{\alpha=n_0+1}^n\o{(V_p)_{\alpha j}}dz^\alpha$, and one has
\begin{equation*}
  \left( \sum_{\alpha,\beta=n_0+1}^{n}R^j_{i\alpha\b{\beta}}dz^\alpha\wedge d\b{z}^\beta\right)_{1\leq j,i\leq r}=A\wedge \o{A}^\top,
\end{equation*}
where $A=(A_{jp})$ is a $r\times N_1$ matrix with $(1,0)$-forms in $V^*_x$ as entries. 

Similarly, by considering the linear map 
\begin{equation*}
  H^U_x:\mr{End}(U_x)\to \mr{End}(E_x^*), \quad  H_x^V(\p_\alpha\otimes dz^\gamma)=R^j_{i  \alpha \bar{\gamma}}  e^i\otimes  e_j. 
\end{equation*}
One can obtain that
\begin{align*}
\begin{split}
&(H^V_x(E_{\alpha\gamma}))_{1\leq \alpha,\gamma\leq n_0}=\left((R_{i \b{j} \alpha \bar{\gamma}})_{1\leq j,i\leq r}\right)_{1\leq \alpha,\gamma\leq n_0} \quad (i\, \alpha\text{ row}, j\,\gamma\text{ column}),
 \end{split}
\end{align*}
which follows that
\begin{equation*}
  \left( \sum_{\alpha,\beta=1}^{n_0}R^j_{i\alpha\b{\beta}}dz^\alpha\wedge d\b{z}^\beta\right)_{1\leq j,i\leq r}=-B\wedge \o{B}^\top,
\end{equation*}
where $B=(B_{jp})$ is a $r\times N_2$ (one can choose $N_2=n_0\cdot r$) matrix with $(0,1)$-forms in $\o{U^*_x}$ as entries. 

Thus, if $(E,h^E)$ is strongly decomposably non-negative of type I, then
for any $x\in X$, the Chern curvature matrix at this point 
has the following form
\begin{equation}\label{cur4}
  R=-B\wedge \o{B}^\top+A\wedge \o{A}^\top
\end{equation}
with respect to a unitary frame, where $B$ is a $r\times N_2$ matrix with  $(0,1)$-forms as entries, $A$ is a $r\times N_1$ matrix  with $(1,0)$-forms as entries, and
\begin{equation}\label{cur-con}
   \mr{span}_{\mb{C}}\{\o{B}\}\cap   \mr{span}_{\mb{C}}\{A\}=\{0\},
\end{equation}
 where $$\{\o{B}\}:=\{\o{B_{ip}}, {1\leq i\leq r,1\leq p\leq N_2}\}$$ and 
  $$\{A\}:=\{A_{ip},1\leq i\leq r, 1\leq p\leq N_1\}.$$
\begin{remark}\label{remark:unitary}
It is noted that the above argument is independent of the choice of unitary frames. If $\tilde{e}=e\cdot a$ is also a unitary frame, then $a$ is a unitary matrix. By \eqref{Tran}, one has
\begin{align*}
\begin{split}
  \wt{R}&=a^{-1}\cdot(-B\wedge \o{B}^\top+A\wedge \o{A}^\top)\cdot a\\
  &=-a^{-1}B\wedge \o{a^{-1}B}^\top+a^{-1}A\wedge \o{a^{-1}A}^\top,
 \end{split}
\end{align*}
which has the form \eqref{cur4}. Moreover, $\mr{span}_{\mb{C}}\{\o{B}\}=\mr{span}_{\mb{C}}\{\o{a^{-1}B}\}$ and $ \mr{span}_{\mb{C}}\{A\}= \mr{span}_{\mb{C}}\{a^{-1}A\}$, \eqref{cur-con} is equivalent to 
\begin{equation*}
  \mr{span}_{\mb{C}}\{\o{a^{-1}B}\}\cap \mr{span}_{\mb{C}}\{a^{-1}A\}=\{0\}.
\end{equation*}
\end{remark}
\vspace{5mm}
Conversely, we assume \eqref{cur4} and \eqref{cur-con} hold. For any $x\in X$, taking local holomorphic coordinates $\{z^\alpha\}_{1\leq\alpha\leq n}$ around $x\in X$ such that 
\begin{equation*}
  \mathrm{span}_{\mb{C}}\{dz^1|_x,\cdots, dz^{n_0}|_x\}=\mr{span}_{\mb{C}}\{\o{B}\}
\end{equation*}
and 
\begin{equation*}
  \mathrm{span}_{\mb{C}}\{dz^{n_0+1}|_x,\cdots, dz^{n_1}|_x\}=\mr{span}_{\mb{C}}\{A\}.
\end{equation*}
Now we set
\begin{equation*}
  U_x:= \mathrm{span}_{\mb{C}}\{\p_1|_x,\cdots, \p_{n_0}|_x\},\quad V_x=\mathrm{span}_{\mb{C}}\{\p_{n_0+1}|_x,\cdots, \p_{n}|_x\}.
\end{equation*}
Then $U_x\oplus V_x=T^{1,0}_xX$.  Using \eqref{cur4} and \eqref{cur-con}, one can check that $(E,h^E)$ is strongly decomposably non-negative.

 Hence $(E,h^E)$ is strongly decomposably non-negative of type I if and only if the Chern curvature matrix of $(E,h^E)$ satisfies  \eqref{cur4} and \eqref{cur-con}.

Next, we assume that $(E,h^E)$ is strongly decomposably positive of type I, i.e., it is strongly decomposably non-negative of type I and 
\begin{equation}\label{con1}
 R_{i\b{j}\alpha\b{\beta}}u^{i\alpha}\o{u^{j\beta}}=0\Longrightarrow u^{i\alpha}=0,\,\text{ for all }1\leq i\leq r,\,1\leq \alpha\leq n_0
\end{equation}
and
\begin{equation}\label{con2}
   R_{i\b{j}\alpha\b{\beta}}v^{\b{j}\alpha}\o{v^{\b{i}\beta}}=0\Longrightarrow v^{\b{j}\alpha}=0,\,\text{ for all }1\leq j\leq r,\,n_0+1\leq \beta\leq n.
\end{equation}
By the equivalent description of strongly decomposably non-negative of type I, i.e., \eqref{cur4} and \eqref{cur-con} hold, one has 
\begin{align*}
\begin{split}
R_{i\b{j}\alpha\b{\beta}}u^{i\alpha}\o{u^{j\beta}}=\sum_{p=1}^{N_2}|B_{ip\b{\alpha}}\o{u^{i\alpha}}|^2.
 \end{split}
\end{align*}
\begin{definition}
Let $B$ be a $r\times N_2$ matrix with  $(0,1)$-forms as entries. We define the following $N_2\times rn_0$-matrix $\mathbf{B}$ as
	\begin{equation}\label{defn B}
  \mathbf{B}:=(B_{ip\b{\alpha}})_{p,i\alpha}=\begin{pmatrix}
  B_{11\o{1}}& B_{11\o{2}}&\cdots& B_{r1\o{n_0}} \\
   B_{12\o{1}}& B_{11\o{2}}&\cdots& B_{r2\o{n_0}}\\
   \vdots&\vdots&\ddots&\vdots\\
    B_{1N_2\o{1}}& B_{1N_2\o{2}}&\cdots& B_{rN_2\o{n_0}}
\end{pmatrix}_{N_2\times rn_0}.
\end{equation}
Similarly, if $A$ is a $r\times N_1$ matrix  with $(1,0)$-forms as entries, we define 
\begin{equation}\label{defn A}
  \mathbf{A}:=(A_{jp\alpha})_{p,j\alpha}=\begin{pmatrix}
  A_{11(n_0+1)}& A_{11(n_0+2)}&\cdots& A_{r1n} \\
   A_{12(n_0+1)}& A_{11(n_0+2)}&\cdots& A_{r2n}\\
   \vdots&\vdots&\ddots&\vdots\\
    A_{1N_1(n_0+1)}& A_{1N_1(n_0+2)}&\cdots& A_{rN_1n}
\end{pmatrix}_{N_1\times r(n-n_0)}.
\end{equation}
\end{definition}

Hence \eqref{con1} is equivalent to the equation $\mathbf{B}x=0$ has only zero solution. This is also equivalent to $\mathrm{rank}(\mathbf{B})=r n_0$. Similarly, \eqref{con2} is equivalent to 
 $\mathrm{rank}(\mathbf{A})=r(n-n_0)$.

In a word, we obtain
\begin{theorem}\label{criterion}
\begin{itemize}
  \item A Hermitian vector bundle $(E,h^E)$ is strongly decomposably non-negative of type I  if and only if 
\eqref{cur4} and \eqref{cur-con} hold.
\item A Hermitian vector bundle $(E,h^E)$ is strongly decomposably positive of type I if and only if 
\eqref{cur4} and \eqref{cur-con} hold, and $\mr{rank}(\mathbf{A})=r\dim V_x$, $\mr{rank}(\mathbf{B})=r\dim U_x$.
\end{itemize}
\end{theorem}
As a result, we obtain the following criteria of (dual) Nakano positive vector bundles. 
\begin{corollary}\label{corNakano}
	\begin{itemize}
  \item A Hermitian vector bundle $(E,h^E)$ is Nakano positive if and only if the Chern curvature matrix has the form 
  \begin{equation*}
  R=-B\wedge \o{B}^\top
\end{equation*}
with respect to some unitary frame, where $B$ is a $r\times N$ matrix with $(0,1)$-forms as entries, and $\mr{rank}(\mathbf{B})=rn$. 
\item A Hermitian vector bundle $(E,h^E)$ is dual Nakano positive if and only if the Chern curvature matrix has the form 
  \begin{equation*}
  R=A\wedge \o{A}^\top
\end{equation*}
with respect to some unitary frame, where $A$ is a $r\times N$ matrix with $(1,0)$-forms as entries, and $\mr{rank}(\mathbf{A})=rn$. 

\end{itemize}

\end{corollary}
\begin{remark}
Note that the above corollary for dual Nakano positivity was previously observed by F. Fagioli \cite[{Page 13, Positivity in [Fin20]}]{Fag20} as a statement without proof.
\end{remark}

\subsection{Weak positivity of Schur forms}

In this subsection, we will prove the weak positivity of Schur forms for strongly decomposably positive vector bundles of type I.

\subsubsection{Schur polynomial}
 Each partition $\lambda \in \Lambda(k, r)$ gives rise to a Schur polynomial $P_\lambda \in \mathbb{Q}\left[c_1, \ldots, c_r\right]$ of degree $k$, defined as $k\times k$ determinant
\begin{equation*}
   P_\lambda\left(c_1, \ldots, c_r\right)=\operatorname{det}\left(c_{\lambda_i-i+j}\right)_{1 \leqslant i, j \leqslant k},
\end{equation*}
where by convention $c_0=1$ and $c_i=0$ if $i \notin[0, r]$. 

Denote by $\mathrm{M}_r(\mathbb{C})$ and $\mathrm{G L}_r(\mathbb{C})$ the vector spaces of $r \times r$ matrices and the general linear group of degree $r$. A map $P: \mathrm{M}_r(\mathbb{C}) \rightarrow \mathbb{C}$ is called $\mathrm{G L}_r(\mathbb{C})$-invariant if it is invariant under the conjugate action of $\mathrm{G L}_r(\mathbb{C})$ on $\mathrm{M}_r(\mathbb{C})$.
Now we define the following $\mathrm{G L}_r(\mathbb{C})$-invariant function $c_i: \mathrm{M}_r(\mathbb{C}) \rightarrow \mathbb{C}, i=1, \ldots, r$ by 
$$
\operatorname{det}\left(I_r+t X\right)=\sum_{i=0}^r t^i \cdot c_i(X),
$$
where $I_r$ is the identity matrix in $\mr{M}_r(\mb{C})$.
Then the graded ring of $\mathrm{G L}_r(\mathbb{C})$-invariant homogeneous polynomials on $\mathrm{M}_r(\mathbb{C})$, which we denote here by $\mathrm{I}(r)=\bigoplus_{k=0}^{+\infty} \mathrm{I}(r)_k$, is multiplicatively generated by $c_1, \ldots, c_r$.

Let $(E,h^E)$ be a Hermitian vector bundle, the $i$-th Chern form $c_i(E,h^E)$ is defined by
$$
c_i\left(E, h^E\right)=c_i\left(\frac{\sqrt{-1} }{2 \pi}R^E\right) .
$$
For each $\lambda\in \Lambda(k,r)$, recall the Schur form (see Introduction) can be given by
\begin{equation*}
  P_\lambda(c(E,h^E))=P_\lambda(c_1(E,h^E),\ldots, c_r(E,h^E)),
\end{equation*}
which represents the {\it Schur class}
\begin{equation*}
  P_\lambda(c(E)):=P_{\lambda}(c_1(E),\ldots, c_r(E))\in \mr{H}^{2k}(X,\mb{Z}).
\end{equation*}

\subsubsection{Griffiths cone}
By \cite[Page 242, (5.6)]{Griffiths}, each $P\in \mathrm{I}(r)_k$ can be written as 
\begin{equation}\label{exp}
  P(B)=\sum_{\sigma, \tau \in S_k} \sum_{\rho \in[1, r]^k} p_{\rho \sigma \tau} B_{\rho_{\sigma(1)} \rho_{\tau(1)}} \cdots B_{\rho_{\sigma(k)} \rho_{\tau(k)}},
\end{equation}
where $B_{\lambda \mu}, \lambda, \mu=1, \ldots, r$ are the components of the matrix $B$, $S_k$ is the permutation group on $k$ indices and $[1, r]:=\{1, \ldots, r\}$.
An element $P\in \mathrm{I}(r)_k$ is called Griffiths non-negative if it can be expressed in the form \eqref{exp} with
$$
p_{\rho \sigma \tau}=\sum_{t \in T} \lambda_{\rho t} \cdot q_{\rho \sigma t} \bar{q}_{\rho \tau t},
$$
for some finite set $T$, some real numbers $\lambda_{\rho t} \geqslant 0$, and complex numbers $q_{\rho \sigma t}$.

The {\it Griffiths cone} $\Pi(r) \subset \mathrm{I}(r)$ is defined as the cone of Griffiths non-negative polynomials.  
\begin{proposition}[{Fulton-Lazarsfeld \cite[Proposition A.3]{FL}}]
Let
$$
P=\sum_{\lambda \in \Lambda(k, r)} a_\lambda(P) P_\lambda \quad\left(a_\lambda(P) \in \mathbb{Q}\right)
$$
be a non-zero weighted homogeneous polynomial in $\mathbb{Q}\left[c_1, \ldots, c_r\right]$. Then $P$ lies in the Griffiths cone $\Pi(r)$ if and only if each of the Schur coefficients $a_\lambda(P)$ is non-negative.
\end{proposition}
In particular, for each $\lambda\in \Lambda(k,r)$, one has
\begin{equation*}
  P_\lambda(B)=\sum_{\sigma, \tau \in S_k} \sum_{\rho \in[1, r]^k} p_{\rho \sigma \tau} B_{\rho_{\sigma(1)} \rho_{\tau(1)}} \cdots B_{\rho_{\sigma(k)} \rho_{\tau(k)}},
\end{equation*}
where $p_{\rho \sigma \tau}=\sum_{1 \leq i, j \leq m} \left(\frac{1}{k!}\right)^2a_{ij}(\tau) \overline{a_{ij}(\sigma)}$ with $(a_{ij}(\tau))\in \mathrm{U}(m)$, see \cite[(A.6)]{FL}. Denote $T=[1,m]^2$ and $q_{\sigma t}:=\o{a_t(\sigma)}$ for any $t\in T$, then 
\begin{equation}\label{exp1}
  P_\lambda(B)=\left(\frac{1}{k!}\right)^2\sum_{\sigma, \tau \in S_k} \sum_{\rho \in[1, r]^k} \left(\sum_{t\in T}q_{\sigma t}\o{q_{\tau t}}\right) B_{\rho_{\sigma(1)} \rho_{\tau(1)}} \cdots B_{\rho_{\sigma(k)} \rho_{\tau(k)}}.
\end{equation}

\subsubsection{Weak positivity of Schur forms}

We assume that $(E,h^E)$ is a strongly decomposably positive vector bundle of type I over a complex manifold $X$. By Theorem \ref{criterion},
for any $x\in X$, there exists a decomposition $T_x^{1,0}X=U_x\oplus V_x$ such that
 the Chern curvature matrix $R$ of $(E,h^E)$ has the form 
\begin{equation}\label{cur5}
  R=-B\wedge\o{B}^\top+A\wedge \o{A}^\top
\end{equation}
with respect to a unitary frame,
where $B$ is a $r\times N$ matrix with  $(0,1)$-forms in $\o{U^*_x}$ as entries, $A$ is a $r\times N$ matrix  with $(1,0)$-forms in $V^*_x$ as entries. Moreover, $\mathrm{rank}(\mathbf{A})=r\cdot \dim V_x$,  $\mathrm{rank}(\mathbf{B})=r\cdot \dim U_x$.

For each $\lambda\in \Lambda(k,r)$, by \eqref{exp1}, the Schur form $P_\lambda(c(E,h^E))$ is given by
\begin{align*}
\begin{split}
   P_\lambda(c(E,h^E))=\left(\frac{\sqrt{-1}}{2\pi}\right)^k\frac{1}{(k!)^2}\sum_{\sigma,\tau\in S_k}\sum_{\rho\in [1,r]^k}\left(\sum_{t\in T}q_{\sigma t}\o{q_{\tau t}}\right)\cdot \bigwedge_{j=1}^k R_{\rho_{\sigma(j)}\o{\rho_{\tau(j)}}}.
 \end{split}
\end{align*}
By \eqref{cur5}, the Chern curvature matrix satisfies
\begin{align*}
\begin{split}
  R_{\rho_{\sigma(j)}\o{\rho_{\tau(j)}}}&=(B_{\rho_{\tau(j)}c_j\b{\beta_j}}\o{B_{\rho_{\sigma(j)}c_j\o{\alpha_j}}}+A_{\rho_{\tau(j)}c_j\alpha_j}\o{A_{\rho_{\sigma(j)}c_j\beta_j}})dz^{\alpha_j}\wedge d\b{z}^{\beta_j}\\
  &=\sum_{c_j=1}^N(\o{B_{\rho_{\sigma(j)}c_j}}\wedge B_{\rho_{\tau(j)}c_j}+A_{\rho_{\tau(j)}c_j}\wedge \o{A_{\rho_{\sigma(j)}c_j}})\\
  &  =\sum_{c_j=1}^N\sum_{\epsilon_j\in\{0,1\}}(\o{B_{\rho_{\sigma(j)}c_j}}\wedge B_{\rho_{\tau(j)}c_j})^{\epsilon_j}\wedge(A_{\rho_{\tau(j)}c_j}\wedge \o{A_{\rho_{\sigma(j)}c_j}})^{1-\epsilon_j},
 \end{split}
\end{align*}
which follows that
\begin{align*}
\begin{split}
&  \bigwedge_{j=1}^k R_{\rho_{\sigma(j)}\o{\rho_{\tau(j)}}}=\bigwedge_{j=1}^k\sum_{c_j=1}^N\sum_{\epsilon_j\in\{0,1\}}(\o{B_{\rho_{\sigma(j)}c_j}}\wedge B_{\rho_{\tau(j)}c_j})^{\epsilon_j}\wedge(A_{\rho_{\tau(j)}c_j}\wedge \o{A_{\rho_{\sigma(j)}c_j}})^{1-\epsilon_j}\\
  &=\sum_{c\in [1,N]^k}\sum_{\epsilon\in \{0,1\}^k}\bigwedge_{j=1}^k(\o{B_{\rho_{\sigma(j)}c_j}}\wedge B_{\rho_{\tau(j)}c_j})^{\epsilon_j}\wedge(A_{\rho_{\tau(j)}c_j}\wedge \o{A_{\rho_{\sigma(j)}c_j}})^{1-\epsilon_j}\\
  &=\sum_{c\in [1,N]^k}\sum_{\epsilon\in \{0,1\}^k}\bigwedge_{j=1}^k(-1)^{\epsilon_j+1}(\o{B_{\rho_{\sigma(j)}c_j}})^{\epsilon_j}\wedge (\o{A_{\rho_{\sigma(j)}c_j}})^{1-\epsilon_j}\wedge (B_{\rho_{\tau(j)}c_j})^{\epsilon_j}\wedge(A_{\rho_{\tau(j)}c_j})^{1-\epsilon_j}\\
  &=\sum_{c\in [1,N]^k}\sum_{\epsilon\in \{0,1\}^k}(-1)^{|\epsilon|+k}(-1)^{\frac{k(k-1)}{2}}\cdot\\
  &\quad (\bigwedge_{j=1}^k(\o{B_{\rho_{\sigma(j)}c_j}})^{\epsilon_j}\wedge (\o{A_{\rho_{\sigma(j)}c_j}})^{1-\epsilon_j})\wedge (\bigwedge_{j=1}^k(B_{\rho_{\tau(j)}c_j})^{\epsilon_j}\wedge(A_{\rho_{\tau(j)}c_j})^{1-\epsilon_j}),
 \end{split}
\end{align*}
where $|\epsilon|:=\sum_{j=1}^k \epsilon_j$.

Recall that $\rho\in [1,r]^k$, $t\in T$, $c\in [1,N]^k$ and $\epsilon\in \{0,1\}^k$, we obtain that
\begin{align*}
\begin{split}
  &P_\lambda(c(E,h^E))=\left(\frac{\sqrt{-1}}{2\pi}\right)^k\frac{1}{(k!)^2}(-1)^{\frac{k(k-1)}{2}}\sum_{\rho,t,c,\epsilon}(-1)^{|\epsilon|+k}\cdot\\
&  (\sum_{\sigma\in S_k}q_{\sigma t}\bigwedge_{j=1}^k(\o{B_{\rho_{\sigma(j)}c_j}})^{\epsilon_j}\wedge (\o{A_{\rho_{\sigma(j)}c_j}})^{1-\epsilon_j})\wedge (\sum_{\tau\in S_k}\o{q_{\tau t}}\bigwedge_{j=1}^k(B_{\rho_{\tau(j)}c_j})^{\epsilon_j}\wedge(A_{\rho_{\tau(j)}c_j})^{1-\epsilon_j}).
 \end{split}
\end{align*}
Now we set
\begin{equation}\label{psi}
  \psi_{\rho t c\epsilon}:=\sum_{\sigma\in S_k}q_{\sigma t}\bigwedge_{j=1}^k(\o{B_{\rho_{\sigma(j)}c_j}})^{\epsilon_j}\wedge (\o{A_{\rho_{\sigma(j)}c_j}})^{1-\epsilon_j},
\end{equation}
which is a $(|\epsilon|,k-|\epsilon|)$-form.
Hence
\begin{align}\label{schur forms}
\begin{split}
  P_\lambda(c(E,h^E))=\left(\frac{1}{2\pi}\right)^k\left(\frac{1}{k!}\right)^2(\sqrt{-1})^{k^2}\sum_{\rho,t,c,\epsilon}(-1)^{|\epsilon|+k}\psi_{\rho t c\epsilon}\wedge \o{\psi_{\rho t c\epsilon}}.
 \end{split}
\end{align}

For any non-zero decomposable $(n-k,0)$-form $\eta=\eta_1\wedge\cdots\wedge\eta_{n-k}$, where $\eta_i,1\leq i\leq n-k$, are $(1,0)$-forms, we assume that $\eta_1,\cdots,\eta_{i_0}\in U_x^*$ and $\eta_{i_0+1},\cdots,\eta_{n-k}\in V_x^*$. 
Now we can take local holomorphic coordinates $\{z^\alpha\}_{1\leq\alpha\leq n}$ around $x\in X$ such that 
\begin{equation*}
  U_x^*= \mathrm{span}_{\mb{C}}\{dz^1|_x,\cdots, dz^{n_0}|_x\},\text{ with } dz^j|_x=\eta_j,\quad 1\leq j\leq i_0
\end{equation*}
and 
\begin{equation*}
  V_x^*= \mathrm{span}_{\mb{C}}\{dz^{n_0+1}|_x,\cdots, dz^{n}|_x\},\text{ with } dz^{n_0-i_0+j}|_x=\eta_j,\quad i_0+1\leq j\leq n-k,
\end{equation*}
and so $\psi_{\rho t c \epsilon}$ can be written in the following form 
\begin{multline*}
	 \psi_{\rho t c \epsilon}=\sum_{1\leq\alpha_1<\cdots<\alpha_{|\epsilon|}\leq n_0\atop n_0+1\leq\beta_1<\cdots<\beta_{k-|\epsilon|}\leq n}\psi_{\alpha_1\cdots\alpha_{|\epsilon|}\b{\beta}_1\cdots\b{\beta}_{k-|\epsilon|}}dz^{\alpha_1}\wedge\cdots\wedge dz^{\alpha_{|\epsilon|}}\\
	 \wedge d\b{z}^{\beta_1}\wedge\cdots\wedge d\b{z}^{\beta_{k-|\epsilon|}}.
\end{multline*}
Then 
\begin{align}\label{schur1}
\begin{split}
& \quad (\sqrt{-1})^{k^2}\sum_{\rho,t,c,\epsilon}(-1)^{|\epsilon|+k}\psi_{\rho t c\epsilon}\wedge \o{\psi_{\rho t c\epsilon}}\wedge (\sqrt{-1})^{(n-k)^2}\eta\wedge\o{\eta}\\
&=(\sqrt{-1})^{k^2}\sum_{\rho,t,c,|\epsilon|=n_0-i_0}(-1)^{n_0-i_0+k}(\sqrt{-1})^{(n-k)^2}dz^1\wedge\cdots \wedge dz^{i_0}\wedge \\&dz^{n_0+1}\wedge\cdots\wedge dz^{n-i_0+n-k}\wedge
d\b{z}^1\wedge \cdots\wedge d\b{z}^{i_0}\wedge d\b{z}^{n_0+1}\wedge\cdots\wedge d\b{z}^{n-i_0+n-k}\wedge\\
&\left(\psi_{i_0+1, \cdots,{n_0},\o{n_0-i_0+n-k+1},\cdots,\b{n}}dz^{i_0+1}\wedge\cdots\wedge dz^{n_0}\wedge d\b{z}^{n_0-i_0+n-k+1}\wedge\cdots\wedge d\b{z}^{n}\right)\wedge\\
&\left(\o{\psi_{i_0+1, \cdots,{n_0},\o{n_0-i_0+n-k+1},\cdots,\b{n}}}d\b{z}^{i_0+1}\wedge\cdots\wedge d\b{z}^{n_0}\wedge d{z}^{n_0-i_0+n-k+1}\wedge\cdots\wedge d{z}^{n}\right)\\
&=\sum_{\rho,t,c,|\epsilon|=n_0-i_0}|\psi_{i_0+1, \cdots,{n_0},\o{n_0-i_0+n-k+1},\cdots,\b{n}}|^2\cdot\\
&\quad (\sqrt{-1})^{n^2}dz^1\wedge \cdots\wedge dz^n\wedge d\b{z}^1\wedge\cdots\wedge d\b{z}^n,
 \end{split}
\end{align}
which is a non-negative volume form. By \eqref{schur forms}, the Schur form $P_\lambda(c(E,h^E))$ is weakly non-negative. 

We will show the weak positivity of Schur form $P_\lambda(c(E,h^E))$ using a proof by contradiction. Specifically, we will derive a contradiction when assuming $P_\lambda(c(E,h^E))\wedge(\sqrt{-1})^{(n-k)^2}\eta\wedge\b{\eta}=0$.

 By \eqref{schur1}, one knows that 
$$P_\lambda(c(E,h^E))\wedge(\sqrt{-1})^{(n-k)^2}\eta\wedge\b{\eta}=0$$ if and only if 
\begin{equation}\label{psi1}
  \psi_{i_0+1, \cdots,{n_0},\o{n_0-i_0+n-k+1},\cdots,\b{n}}=0
\end{equation}
for any $\rho,t,c,|\epsilon|=n_0-i_0$.

Now we take a special vector $\epsilon=(\underbrace{1,\cdots,1}_{n_0-i_0},0\cdots,0)$ and denote $j_0=n_0-i_0$, by \eqref{psi}, then 
\begin{align*}
  \psi_{\rho t c\epsilon}&=\sum_{\sigma \in S_k} q_{\sigma t}\o{B_{\rho_{\sigma(1)}c_1}}\wedge \cdots\wedge \o{B_{\rho_{\sigma(j_0)}c_{j_0}}}\wedge \o{A_{\rho_{\sigma(j_0+1)}c_{j_0+1}}}\wedge\cdots \wedge \o{A_{\rho_{\sigma(k)}c_k}}.
\end{align*}
Combining with \eqref{psi1}, one has
\begin{align}\label{psi2}
\begin{split}
 0 &=\psi_{i_0+1, \cdots,{n_0},\o{n_0-i_0+n-k+1},\cdots,\b{n}}\\
  &=\sum_{\tau_1\in S_{j_0}}\sum_{\tau_2\in S_{k-j_0}}\sum_{\sigma \in S_k} \mr{sgn}(\tau_1)\mr{sgn}(\tau_2)q_{\sigma t}\cdot \\
  &\o{B_{\rho_{\sigma(1)}c_1\o{\tau_1(i_0+1)}}}\cdots\o{B_{\rho_{\sigma(j_0)}c_{j_0}\o{\tau_1(n_0)}}}\cdot \o{A_{\rho_{\sigma(j_0+1)}c_{j_0+1}{\tau_2(j_0+n-k+1)}}}\cdots \o{A_{\rho_{\sigma(k)}c_k{\tau_2(n)}}}.
 \end{split}
\end{align}
Since $\mathrm{rank}(\mathbf{A})=r\cdot \dim V_x$,  $\mathrm{rank}(\mathbf{B})=r\cdot \dim U_x$, without loss of generality, we assume that the submatrices
$$\mathbf{B}'=(\mathbf{B}_{c,i\alpha})_{1\leq c\leq r\dim U_x, 1\leq i\leq r,1\leq \alpha\leq \dim U_x}$$
and 
\begin{equation*}
  \mathbf{A}'=(\mathbf{A}_{c,j\alpha})_{1\leq c\leq r\dim V_x, 1\leq j\leq r,1\leq \alpha\leq \dim V_x}
\end{equation*}
of $\mathbf{B}$ and $\mathbf{A}$ are inverse. By \eqref{psi2} and note that 
$\mathbf{B}'_{c,i\alpha}=B_{ic\b{\alpha}}$ and $\mathbf{A}'_{c,j\alpha}=A_{jc\alpha}$, one has
\begin{align}\label{psi3}
\begin{split}
  0&=\sum_{\tau_1\in S_{j_0}}\sum_{\tau_2\in S_{k-j_0}}\sum_{c_1,\cdots,c_{j_0}=1}^{r n_0}\sum_{c_{j_0+1},\cdots,c_{k}=1}^{r (n-n_0)}\sum_{\sigma \in S_k} \mr{sgn}(\tau_1)\mr{sgn}(\tau_2)q_{\sigma t}\cdot \\
  &\o{\mathbf{B}'_{c_1,\rho_{\sigma(1)}{\tau_1(i_0+1)}}}\cdots\o{\mathbf{B}'_{c_{j_0},\rho_{\sigma(j_0)}{\tau_1(n_0)}}}\cdot \o{\mathbf{A}'_{c_{j_0+1},\rho_{\sigma(j_0+1)}{\tau_2(j_0+n-k+1)}}}\cdots \o{\mathbf{A}'_{c_k,\rho_{\sigma(k)}{\tau_2(n)}}}\cdot\\
  &\o{\mathbf{A}'^{-1}_{l_k\beta_k,c_k}}\cdots \o{\mathbf{A}'^{-1}_{l_{j_0+1}\beta_{j_0+1},c_{j_0+1}}}\o{\mathbf{B}'^{-1}_{l_{j_0}\beta_{j_0},c_{j_0}}}\cdots \o{\mathbf{B}'^{-1}_{l_{1}\beta_{1},c_{1}}}\\
  &=\sum_{\tau_1\in S_{j_0}}\sum_{\tau_2\in S_{k-j_0}}\sum_{\sigma \in S_k} \mr{sgn}(\tau_1)\mr{sgn}(\tau_2)q_{\sigma t}\cdot\delta_{\rho_{\sigma(k)}l_k}\delta_{\tau_2(n)\beta_k}\cdots\delta_{\rho_{\sigma(j_0+1)}l_{j_0+1}}\cdot \\
  &\delta_{\tau_2(j_0+n-k+1)\beta_{j_0+1}}\delta_{\rho_{\sigma(j_0)}l_{j_0}}\delta_{\tau_1(n_0)\beta_{j_0}}\cdots\delta_{\rho_{\sigma(1)}l_1}\delta_{\tau_1(i_0+1)\beta_1}
 \end{split}
\end{align}
for any $(\beta_1,\cdots,\beta_{j_0})\in [1,n_0]^{j_0}$, $(\beta_{j_0+1},\cdots,\beta_k)\in [n_0+1,n]^{n-j_0}$ and $(l_1,\cdots,l_k)\in [1,r]^k$. 

By taking 
\begin{equation*}
  \beta_s= \begin{cases}
 	i_0+s& 1\leq s\leq j_0 ,\\
 	n-k+s& j_0+1\leq s\leq k,
 \end{cases}
\end{equation*}
then \eqref{psi3} becomes 
\begin{equation}\label{psi4}
  \sum_{\sigma\in S_k} q_{\sigma t}\delta_{\rho_{\sigma(1)}l_1}\cdots\delta_{\rho_{\sigma(k)}l_k}=0
\end{equation}
for any $\rho,l\in [1,r]^k$ and $t\in T$. 
\begin{remark}\label{rem:equ}
Note that \eqref{psi4} holds if and only if $\psi_{\rho t c \epsilon}=0$ for any $\rho,t,c,\epsilon$. In fact, if $\psi_{\rho t c \epsilon}=0$, then \eqref{psi1} holds, and follows \eqref{psi4}. Conversely, if  \eqref{psi4} holds, then 
\begin{align*}
\begin{split}
  \psi_{\rho t c \epsilon}=\sum_{l\in[1,r]^k} (\sum_{\sigma\in S_k} q_{\sigma t}\delta_{\rho_{\sigma(1)}l_1}\cdots\delta_{\rho_{\sigma(k)}l_k})\bigwedge_{j=1}^k\o{B_{l_jc_j}}^{\epsilon_j}\wedge \o{A_{l_jc_j}}^{1-\epsilon_j}=0.
 \end{split}
\end{align*}
\end{remark}
\vspace{5mm}
For $k\leq r$, one can take $\rho_i=l_i=i$ for $1\leq i\leq k$. Thus
	\begin{align*}
\begin{split}
 0=\sum_{\sigma\in S_k} q_{\sigma t}\delta_{\rho_{\sigma(1)}l_1}\cdots\delta_{\rho_{\sigma(k)}l_k}=q_{\mathrm{Id},t}
 \end{split}
\end{align*}
for any $t\in T$, which is a contradiction since $(q_{\mathrm{Id},t})_{t\in T}\in \mathrm{U}(m)$ is a unitary matrix. Hence all $(k,k)$-Schur forms $P_\lambda(c(E,h^E))$ are weakly positive for any $k\leq r$. In particular, all Chern forms $c_i(E,h^E),1\leq i\leq r$, are weakly positive. 

For general $k$ and $r$, we take $l_1,\cdots,l_k$ in \eqref{psi4} to be 
\begin{equation*}
 l_i=\rho_i, \text{ for }1\leq i\leq k,
\end{equation*}
and \eqref{psi4} implies that
\begin{equation}\label{psi5}
  \sum_{\rho\in [1,r]^k}\sum_{\sigma\in S_k}\chi_\lambda (\sigma)\delta_{\rho_1\rho_{\sigma(1)}}\cdots\delta_{\rho_k\rho_{\sigma(k)}}=0,
\end{equation}
where 
$$\chi_\lambda(\sigma)=\mathrm{Tr}(\o{q_{\sigma t}})=\sum_{i=1}^m a_{ii}(\sigma)$$ 
is the character of the representation $\phi_\lambda(\sigma)=(a_{ij}(\sigma))\in \mr{U}(m)$ corresponding to the partition $\lambda$. From \cite[(A.5)]{FL}, \eqref{psi5} is equivalent to 
\begin{equation}\label{psi6}
  P_\lambda (I_r)=0.
\end{equation}
Here $P_\lambda(\bullet)$ denotes the  invariant
polynomial corresponding to the Schur function $P_\lambda$ under the isomorphism
$\mathrm{I}(r)\cong \mb{Q}(c_1,\cdots,c_r)$.

Denote by  $x_1,\cdots,x_r$ the Chern roots, which are defined by
\begin{equation*}
  \sum_{j=0}^r c_jt^j=(1+tx_1)(1+tx_2)\cdots(1+tx_r).
\end{equation*}
Recall that the Schur polynomial is defined by
\begin{align*}
\begin{split}
    P_\lambda\left(c_1, \ldots, c_r\right)&=\operatorname{det}\left(c_{\lambda_j-j+l}\right)_{1 \leqslant j, l \leqslant k},
 \end{split}
\end{align*}
where $\lambda=(\lambda_1,\cdots,\lambda_k)\in \Lambda(k,r)$ is a partition  satisfying
 $$\sum_{i=1}^k\lambda_i=k \text{ and } r\geq \lambda_1\geq\cdots\geq\lambda_k\geq 0.$$ Denote by $\lambda'$ the conjugate partition to the partition $\lambda$, see e.g. \cite[Section 4.1, Page 45]{FH}, then 
\begin{equation}\label{conj-part}
  \lambda'=(\lambda_1',\cdots,\lambda_r'),\quad\text{ with } \sum_{i=1}^r\lambda'_i=k\text{ and } \lambda'_1\geq\cdots\geq\lambda'_r\geq 0.
\end{equation}
 The second Jacobi-Trudi identity (or Giambell's formula) gives 
\begin{align}\label{Schur-equ}
\begin{split}
    P_\lambda\left(c_1, \ldots, c_r\right)&=\operatorname{det}\left(c_{\lambda_j-j+l}\right)_{1 \leqslant j, l \leqslant k}=s_{\lambda'}(x_1,\cdots,x_r),
 \end{split}
\end{align}
see e.g. \cite[Page 455, (A.6)]{FH},
where 
\begin{equation}\label{Schur function}
  s_{\lambda'}(x_1,\cdots,x_r):=\frac{\left|\begin{array}{cccc}x_1^{\lambda'_1+r-1} & x_2^{\lambda'_1+r-1} & \ldots & x_r^{\lambda'_1+r-1} \\ x_1^{\lambda'_2+r-2} & x_2^{\lambda'_2+r-2} & \ldots & x_r^{\lambda'_2+r-2} \\ \vdots & \vdots & \ddots & \vdots \\ x_1^{\lambda'_r} & x_2^{\lambda'_r} & \ldots & x_r^{\lambda'_r}\end{array}\right|}{\prod_{1 \leq i<j \leq r}\left(x_i-x_j\right)}.
\end{equation}
In particular, we have
\begin{align}\label{contr}
\begin{split}
 P_\lambda(I_r)&=P_\lambda\left(c_1=C^1_r,\cdots, c_i=C_r^i,\cdots,c_r=C^r_r\right)\\
 &=s_{\lambda'}(1,\cdots,1)\\
 &=\prod_{1 \leq i<j \leq r} \frac{\lambda'_i-\lambda'_j+j-i}{j-i}\geq 1,
 \end{split}
\end{align}
where the third equality follows from \cite[Page 461, (ii)]{FH}. Hence, $ P_\lambda(I_r)\neq 0$, which contradicts to \eqref{psi6}. Thus,
$$P_\lambda(c(E,h^E))\wedge(\sqrt{-1})^{(n-k)^2}\eta\wedge\b{\eta}>0$$
for any non-zero decomposable $(n-k,0)$-form $\eta$, and so $P_\lambda(c(E,h^E))$ is a weakly positive $(k,k)$-form. 

In a word, we obtain
\begin{theorem}
	Let $(E,h^E)$ be a strongly decomposably positive vector bundle of type I over a complex manifold $X$, $\mathrm{rank}E=r$, and $\dim X=n$. Then the Schur form 
	$P_\lambda(c(E,h^E))$ is weakly positive for any partition $\lambda\in \Lambda(k,r)$, $k\leq n$ and $k\in\mb{N}$.
\end{theorem}

In particular, if $(E,h^E)$ is Nakano positive, then the Chern curvature matrix has the form 
\begin{equation*}
  R=-B\wedge\o{B}^\top.
\end{equation*}
By considering $A=0$ in \eqref{psi}, then 
\begin{equation*}
  \psi_{\rho t c \epsilon_1}=\sum_{\sigma\in S_k}q_{\sigma t}\bigwedge_{j=1}^k\o{B_{\rho_{\sigma(j)}c_j}},\quad \epsilon_1=(1,\cdots,1)
\end{equation*}
and 
\begin{equation*}
 \psi_{\rho t c \epsilon}=0,\text{ for any }\epsilon\neq \epsilon_1,
\end{equation*}
By \eqref{schur forms}, one has
\begin{align}\label{schur forms 1}
\begin{split}
  P_\lambda(c(E,h^E))=\left(\frac{1}{2\pi}\right)^k\left(\frac{1}{k!}\right)^2(\sqrt{-1})^{k^2}\sum_{\rho,t,c}\psi_{\rho t c \epsilon_1}\wedge \o{\psi_{\rho t c\epsilon_1}},
 \end{split}
\end{align}
 where $\psi_{\rho t c\epsilon_1}$ is a $(k,0)$-form. For any non-zero $(n-k,0)$-form $\eta$, one has
\begin{align*}
\begin{split}
  & \quad P_\lambda(c(E,h^E))\wedge (\sqrt{-1})^{(n-k)^2}\eta\wedge\o{\eta}\\
  &=\left(\frac{1}{2\pi}\right)^k\left(\frac{1}{k!}\right)^2(\sqrt{-1})^{n^2}\sum_{\rho,t,c}\psi_{\rho t c \epsilon_1}\wedge\eta\wedge  \o{\psi_{\rho t c\epsilon_1}\wedge \eta},
 \end{split}
\end{align*}
which is a non-negative volume form, and so $ P_\lambda(c(E,h^E))$ is non-negative. Moreover, 
\begin{equation}\label{Schur-van}
  P_\lambda(c(E,h^E))\wedge (\sqrt{-1})^{(n-k)^2}\eta\wedge\o{\eta}=0
\end{equation}
if and only if 
\begin{equation}\label{psivan}
  \psi_{\rho t c \epsilon_1}\wedge \eta=0
\end{equation}
for any $\rho\in [1,r]^k$, $t\in T$, $c\in [1,N]^k$.
By the expression of $\psi_{\rho t c\epsilon_1}$, \eqref{psivan} becomes
\begin{align}\label{psivan1}
\begin{split}
  0&=\psi_{\rho t c \epsilon_1}\wedge \eta\\
   &=\sum_{\sigma\in S_k} q_{\sigma t}\o{B_{\rho_{\sigma(1)}c_1}}\wedge \cdots\wedge \o{B_{\rho_{\sigma(k)}c_k}}\wedge \eta\\
  &=\sum_{\sigma\in S_k} q_{\sigma t}\o{B_{\rho_{\sigma(1)}c_1\b{\alpha}_1}} \cdots \o{B_{\rho_{\sigma(k)}c_k\b{\alpha}_k}} dz^{\alpha_1}\wedge\cdots\wedge dz^{\alpha_k}\wedge \eta.
 \end{split}
\end{align}
Since $(E,h^E)$ is Nakano positive, by Corollary \ref{corNakano}, we can take $B$ such that $\mathbf{B}$ is invertible. Multiplying \eqref{psivan1} by $(\mathbf{B}^{-1})_{c_1,l_1\beta_1}\cdots (\mathbf{B}^{-1})_{c_k,l_k\beta_k}$ and summing on $c_1,\cdots,c_k$, one has
\begin{equation*}
\left( \sum_{\sigma\in S_k} q_{\sigma t}\delta_{\rho_{\sigma(1)}l_1}\cdots\delta_{\rho_{\sigma(k)}l_k}\right)dz^{\beta_1}\wedge \cdots\wedge dz^{\beta_k}\wedge \eta=0
\end{equation*}
for any $l=(l_1,\cdots,l_k)\in [1,r]^k$ and $\beta=(\beta_1,\cdots,\beta_k)\in [1,n]^k$. By choosing $\beta_1,\cdots,\beta_k$ such that $dz^{\beta_1}\wedge \cdots\wedge dz^{\beta_k}\wedge \eta\neq 0$, so 
\begin{equation}\label{psi7}
  \sum_{\sigma\in S_k} q_{\sigma t}\delta_{\rho_{\sigma(1)}l_1}\cdots\delta_{\rho_{\sigma(k)}l_k}=0,
\end{equation}
which is exactly \eqref{psi4}. By Remark \ref{rem:equ}, \eqref{psi7} is equivalent to $\psi_{\rho t c\epsilon_1}=0$. Hence \eqref{Schur-van} is equivalent to \eqref{psi7},
which follows that $P_\lambda(I_r)=0$, see \eqref{psi6}. By \eqref{contr}, $P_\lambda(I_r)\neq 0$, so we get a contradiction. Thus, $P_\lambda(c(E,h^E))$ is a positive $(k,k)$-form. 
Similarly, if  $(E,h^E)$ is dual Nakano positive, then $P_\lambda(c(E,h^E))$ is also a positive $(k,k)$-form.

 Hence, we can give an algebraic proof of the following positivity of Schur forms for (dual) Nakano positive vector bundles. 
\begin{theorem}[{Finski \cite[Theorem 1.1]{Fin}}]
	Let $(E, h^E)$  be a (dual) Nakano positive (respectively non-negative) vector bundle of rank $r$ over a complex manifold $X$ of dimension $n$. Then for any $k \in \mathbb{N}$, $k \leqslant n$, and $\lambda \in \Lambda(k, r)$, the $(k, k)$-form $P_\lambda\left(c\left(E, h^E\right)\right)$ is positive (respectively non-negative).
\end{theorem}
\begin{remark}
Note that S. Finski proved the above positivity of Schur forms by using the following two steps: the first one is a refinement of the determinantal formula of Kempf-Laksov on the level of differential forms, which expresses Schur forms as a certain pushforward of the
top Chern form of a Hermitian vector bundle obtained as a quotient of the tensor power of $(E,h^E)$, and the second one is to show the positivity of the top Chern form of a (dual) Nakano positive vector bundle. Our method here is an algebraic proof by analyzing the vanishing of Schur forms, which is very different from S. Finski's approach. 
\end{remark}

\section{Strongly decomposable positivity of type II}\label{strongly positive II}

In this section, we will consider the strongly decomposable positivity of type II,
 which is the direct sum of Nakano positive and dual Nakano positive vector bundles point-wise.  

\subsection{A criterion of type II positivity}

Let $(E,h^E)$ be a strongly decomposably positive vector bundle of type II, see Definition \ref{stongly-II}. 
 By Corollary \ref{corNakano}, with respect to a unitary frame $\{e_1,\cdots,e_{r_1}\}$ of $(E_{1,x},h^E|_{E_{1,x}})$, and a unitary frame $\{e_{r_1+1},\cdots,e_r\}$ of $(E_{2,x},h^E|_{E_{2,x}})$ at $x\in X$, one has
\begin{equation*}
  R^{E}_x|_{E_{1,x}}=-B_1\wedge\o{B_1}^\top, \quad R^{E}_x|_{E_{2,x}}=A_2\wedge \o{A_2}^\top,
\end{equation*}
with $\mathrm{rank}(\mathbf{B}_1)=\dim(E_{1,x})\cdot n$ and $\mathrm{rank}(\mathbf{A}_2)=\dim(E_{2,x})\cdot n$, where $B_1=((B_1)_{ip})_{1\leq i\leq r_1,1\leq j\leq N_1}$ is a matrix with $(0,1)$-forms as entries 
and $A_2=((A_2)_{ip})_{r_1+1\leq i\leq r,1\leq j\leq N_2}$ is a matrix with 
$(1,0)$-forms as entries. The matrices $\mathbf{B}_1$ and $\mathbf{A}_2$ are defined in \eqref{defn B} and \eqref{defn A} respectively.
Now we define the matrices $A_{r\times N}$ and $B_{r\times N}$ ($N=\max\{N_1,N_2\}$) by
\begin{equation}\label{AB}
  B=\begin{pmatrix}
  B_1& 0\\
  0& 0
\end{pmatrix},\quad A=\begin{pmatrix}
  0&0 \\
  A_2&0 
\end{pmatrix}.
\end{equation}
Then
\begin{equation*}
  \mr{rank}(\mathbf{B})=\mathrm{rank}(\mathbf{B}_1)=r_1\cdot n,\quad \mr{rank}(\mathbf{A})=\mathrm{rank}(\mathbf{A}_2)=(n-r_1)\cdot n
\end{equation*}
and
\begin{equation*}
  R^E_x=\begin{pmatrix}
  -B_1\wedge \o{B_1}^\top&0 \\
  0& A_2\wedge\o{A_2}^\top
\end{pmatrix}=-B\wedge\o{B}^\top+A\wedge \o{A}^\top.
\end{equation*}
For the matrix $\mathbf{B}=(B_{ip\b{\alpha}})_{i\alpha,p}$, we can associate it with another matrix $\mathcal{B}$ by
\begin{equation*}
   \mathcal{B}=(\mathcal{B}_{\alpha p})=\left(\sum_{i=1}^rB_{ip\b{\alpha}}e_i\right)_{\alpha p},
\end{equation*}
 which is a $n\times N$ matrix with elements of $E$ as entries. Similarly, we can define a $n\times N$ matrix by 
 \begin{equation*}
  \mathcal{A}=(\mathcal{A}_{\alpha p})=\left(\sum_{i=1}^r A_{ip\alpha}e_i\right)_{\alpha p}.
\end{equation*}
We define 
\begin{equation*}
\{\mc{B}\}:=\left\{\sum_{i=1}^rB_{ip\b{\alpha}}e_i,1\leq p\leq N, 1\leq\alpha\leq n\right\}
\end{equation*}
 and 
 \begin{equation*}
 \{\mathcal{A}\}:=\left\{\sum_{i=1}^rA_{ip{\alpha}}e_i,1\leq p\leq N, 1\leq\alpha\leq n\right\}.	
 \end{equation*}
By the definitions of the matrices $A$ and $B$, one has
\begin{equation*}
  \mathrm{span}_{\mb{C}}\{\mc{A}\}\perp \mathrm{span}_{\mb{C}}\{\mc{B}\}.
\end{equation*}
Hence, if $(E,h^E)$ is a strongly decomposably positive vector bundle of type II, then there are two $r\times N$-matrices $A$, $B$ of $(1,0)$-forms and $(0,1)$-forms respectively, such that with respect to a unitary frame $\{e_i\}_{1\leq i\leq r}$ of $E_x$,
\begin{equation}\label{4.2}
  R^E_x=-B\wedge \o{B}^\top+A\wedge\o{A}^\top,
\end{equation}
and 
\begin{equation}\label{4.3}
  \mathrm{span}_{\mb{C}}\{\mc{A}\}\perp \mathrm{span}_{\mb{C}}\{\mc{B}\}.
\end{equation}
Moreover, the ranks of $\mathbf{A}$ and $\mathbf{B}$ satisfy 
\begin{equation}\label{4.4}
   \mathrm{rank}(\mathbf{B})=r_1\cdot n, \quad \mathrm{rank}(\mathbf{A})=(r-r_1)\cdot n.
\end{equation}
\begin{remark}\label{remark:unitary2}
	If we consider a new unitary frame $\wt{e}=e\cdot a$ for a unitary matrix $a\in \mathrm{U}(r)$, by Remark \ref{remark:unitary}, one has
	\begin{equation}\label{4.1}
  \wt{R^E_x}=-\wt{B}\wedge \o{\wt{B}}^\top+\wt{A}\wedge\o{\wt{A}}^\top,
\end{equation}
with $\wt{B}=a^{-1}\cdot B$ and $\wt{A}=a^{-1}\cdot A$. Moreover, one has
\begin{align*}
\begin{split}
  \wt{\mc{B}}_{\alpha p}=\sum_{i=1}^r\wt{B}_{ip\b{\alpha}}\wt{e}_i=\sum_{i,j=1}^r(a^{-1})_{ij}B_{jp\b{\alpha}}\wt{e}_i=\sum_{j=1}^r B_{jp\b{\alpha}}{e}_j=\mc{B}_{\alpha p}.
 \end{split}
\end{align*}
Similarly, $\wt{\mc{A}}_{\alpha p}=\mc{A}_{\alpha p}$. Hence $\mc{A}$ and $\mc{B}$ are independent of the unitary frame. One can also check that 
\begin{equation*}
  \mathrm{rank}(\wt{\mathbf{B}})=\mathrm{rank}(\mathbf{B})=r_1\cdot n,\quad \mathrm{rank}(\wt{\mathbf{A}})=\mathrm{rank}(\mathbf{A})=(r-r_1)\cdot n.
\end{equation*}
In a word, we show that \eqref{4.2}--\eqref{4.4} hold for any unitary frame. 
\end{remark}
 
 Conversely, we assume that \eqref{4.2}--\eqref{4.4} hold for some unitary frame of $E_x$, $x\in X$. Set
 \begin{equation*}
  E_{1,x}:=\mathrm{span}_{\mb{C}}\{\mc{B}\},\quad E_{2,x}:=\mathrm{span}_{\mb{C}}\{\mc{A}\}.
\end{equation*}
Let $\{e_1,\cdots, e_{r'_1}\}$ be a unitary frame of $E_{1,x}$ and $\{e_{r_1'+1},\cdots, e_{r'}\}$ be a unitary frame of $E_{2,x}$. Since $E_{1,x}\perp E_{2,x}$, so 
$\{e_i\}_{1\leq i\leq r'}$ is a unitary frame of $E_{1,x}\oplus E_{2,x}$. Now we can extend the frame $\{e_i\}_{1\leq i\leq r'}$ and get a unitary frame $\{e_i\}_{1\leq i\leq r}$ of $E_x$.
 By Remark \ref{remark:unitary2}, \eqref{4.2}--\eqref{4.4} also hold for this unitary frame $\{e_i\}_{1\leq i\leq r}$.
Hence
\begin{equation*}
  R_{i\b{j}\alpha\b{\beta}} e_j\otimes \o{e_i}=\sum_{p=1}^N(-\mc{B}_{\beta p}\otimes \o{\mc{B}_{\alpha p}}+\mc{A}_{\alpha p}\otimes \o{\mc{A}_{\beta p}}).
\end{equation*}
So
\begin{equation*}
  R^E_x|_{E_{1,x}}=-B\wedge\o{B}^\top,\,\, R^E_x|_{E_{2,x}}=A\wedge \o{A}^\top
\end{equation*}
and 
\begin{equation*}
  R_{i\b{j}\alpha\b{\beta}}=0\text{ for any } (i,j) \text{ or } (j,i)\in [1,r_1']\times [r_1'+1,r].
\end{equation*}
By \eqref{4.4}, one has
\begin{equation*}
  r_1'\geq r_1,\quad r-r_1'\geq r'-r'_1\geq r-r_1,
\end{equation*}
which follows that
\begin{equation*}
  r_1=r_1',\quad r'=r.
\end{equation*}
Hence $E_x=E_{1,x}\oplus E_{2,x}$. By Corollary \ref{corNakano}, we obtain that
 $R_{i\b{j}\alpha\b{\beta}}u^{i\alpha}\o{u^{j\beta}}>0$ for any non-zero $u=u^{i\alpha} e_i\otimes \p_\alpha\in E_{1,x}\otimes T^{1,0}_xX$, $R_{i\b{j}\alpha\b{\beta}}v^{i\b{\beta}}\o{v^{j\b{\alpha}}}>0$ for any non-zero $v=v^{i\b{\beta}}e_i\otimes \p_{\b{\beta}}\in E_{2,x}\otimes T^{0,1}_xX$. Thus, $(E,h^E)$ is a strongly decomposably positive vector bundle of type II. 
 
 In a word, we obtain a criterion of a strongly decomposably positive vector bundle of type II. 
\begin{theorem}\label{prop-criterion}
	$(E,h^E)$ is a strongly decomposably positive vector bundle of type II if and only if it satisfies \eqref{4.2}--\eqref{4.4}.
\end{theorem}

\subsection{Positivity of Schur forms}

In this subsection, we will consider the positivity of Schur forms for strongly decomposably positive vector bundles of type II. 

Let $E$ and $F$ be two holomorphic vector bundles over a complex manifold $X$, $\mathrm{rank}(E)=r$ and $\mathrm{rank}(F)=q$. Let $x_1,\cdots,x_{r}$ denote the Chern roots of $E$. For any partition $\lambda'$ satisfying \eqref{conj-part}, we denote 
\begin{equation*}
  s_{\lambda'}(c(E)):=s_{\lambda'}(x_1,\cdots,x_r)\in \mathrm{H}^{2k}(X,\mb{R}),
\end{equation*}
where $s_{\lambda'}(x_1,\cdots,x_r)$ is defined in \eqref{Schur function}, which is also called a Schur class. Similarly, one can define the cohomology classes $s_{\lambda'}(c(F))$ and $s_{\lambda'}(c(E\oplus F))$. For these cohomology classes, by Littlewood-Richardson rule, see \cite[Proposition 3.3 (3.14)]{MR4338228}, one has
\begin{equation*}
  s_{\lambda'}(c(E\oplus F))=\sum_{\mu',\nu'}c^{\lambda'}_{\mu'\nu'}s_{\mu'}(c(E))s_{\nu'}(c(F)),
\end{equation*}
where $c^{\lambda'}_{\mu'\nu'}$ is a Littlewood-Richardson coefficient. One can refer to \cite[Chapter 5]{MR1464693} for more details on the Littlewood-Richardson coefficients. By \eqref{Schur-equ}, the Schur class $P_\lambda(c(E\oplus F))$ of the direct sum $E\oplus F$ satisfies
\begin{equation}\label{4.5}
  P_\lambda(c(E\oplus F))=\sum_{\mu',\nu'}c^{\lambda'}_{\mu'\nu'}P_\mu(c(E))P_\nu(c(F))=\sum_{\mu,\nu}c^\lambda_{\mu\nu}P_\mu(c(E))P_\nu(c(F)),
\end{equation}
where $\lambda,\mu,\nu$ are the conjugate partitions to $\lambda',\mu',\nu'$ respectively, the last equality follows from the conjugation symmetry $c^{\lambda'}_{\mu'\nu'}=c^\lambda_{\mu\nu}$, see e.g. \cite[Page 115]{MR1904379}. 

Let $h^E$ and $h^F$ be Hermitian metrics on $E$ and $F$, respectively. The direct sum $E\oplus F$ is equipped with the natural metric $h^E\oplus h^F$. Now we can prove \eqref{4.5} on the level of differential forms. 
\begin{proposition}\label{prop:product}
For any $\lambda\in \Lambda(k,r)$, one has
\begin{equation*}
  P_\lambda(c(E\oplus F,h^E\oplus h^F))=\sum_{\mu,\nu}c^\lambda_{\mu\nu}P_\mu(c(E,h^{E}))\wedge P_\nu(c(F,h^{F})).
\end{equation*}
\end{proposition}
\begin{proof}
We will follow the method in the proofs of \cite[Proposition 3.1]{MR2932990} and \cite[Theorem 3.5]{DF}.
From the definition of total Chern form, one has
\begin{equation*}
  c(E\oplus F,h^E\oplus h^F)=c(E,h^E)\wedge c(F,h^F).
\end{equation*}
Recall that $P_\lambda\left(c_1, \ldots, c_r\right)=\operatorname{det}\left(c_{\lambda_i-i+j}\right)_{1 \leqslant i, j \leqslant k}$, so that 
\begin{align*}
  &\quad  P_\lambda(c(E\oplus F,h^E\oplus h^F))- \sum_{\mu,\nu}c^\lambda_{\mu\nu}P_\mu(c(E,h^{E}))\wedge P_\nu(c(F,h^{F}))\\
  =&\sum_{i_1+2i_2+\cdots+ri_r\atop+j_1+2j_2+\cdots+qj_q=k}f_{i_1\cdots i_r j_1\cdots j_{q}}c_1(E,h^E)^{i_1}\wedge \cdots\wedge c_r(E,h^E)^{i_r}\\
   &\wedge c_1(F,h^F)^{j_1}\wedge \cdots\wedge c_q(F,h^F)^{j_q}.
\end{align*}
where the universal coefficients $f_{i_1\cdots i_rj_1\dots j_q}$ do not depend on $E,F$ and $X$, just depend $r,q$, $P_\lambda$. By \eqref{4.5}, then the cohomology class satisfies
\begin{align*}
\begin{split}
 &\left[\sum_{i_1+2i_2+\cdots+ri_r\atop+j_1+2j_2+\cdots+qj_q=k}f_{i_1\cdots i_r j_1\cdots j_{q}}c_1(E,h^E)^{i_1}\wedge \cdots\wedge c_r(E,h^E)^{i_r}\right.\\
 &\left.\wedge c_1(F,h^F)^{j_1}\wedge \cdots\wedge c_q(F,h^F)^{j_q}\right]=0\\
 \end{split}
\end{align*}
Now we can take $X$ as any $n$-dimensional projective manifold and fix an ample line bundle $A$ on $X$. Let $\omega_A$ be a metric on $A$ with positive curvature. For $m_1,\cdots,m_r,m_{r+1},\cdots,m_{r+q}$ positive integers, we define
\begin{equation*}
  E=A^{\otimes m_1}\oplus \cdots\oplus A^{\otimes m_r},\quad F=A^{\otimes m_{r+1}}\oplus\cdots\oplus A^{\otimes m_{r+q}}.
\end{equation*}
By the same proof as in \cite[ Page 14]{DF}, one can show all universal coefficients $f_{i_1\cdots i_r j_1\cdots j_{q}}$ vanish, which follows that
\begin{equation*}
  P_\lambda(c(E\oplus F,h^E\oplus h^F))- \sum_{\mu,\nu}c^\lambda_{\mu\nu}P_\mu(c(E,h^{E}))\wedge P_\nu(c(F,h^{F}))=0,
\end{equation*}
which completes the proof.
\end{proof}

 Now we assume $(E,h^E)$ is a strongly decomposably positive vector bundle of type II, for any $x\in X$, there exists an orthogonal decomposition of $E_x$,
 \begin{equation*}
  E_x=E_{1,x}\oplus E_{2,x},
\end{equation*}
and the Chern curvature $R^E_x$ has the form 
\begin{equation*}
  R^E_x=\begin{pmatrix}
  R^{E}_x|_{E_{1,x}}&0 \\
  0& R^{E}_x|_{E_{2,x}}
\end{pmatrix}.
\end{equation*}
Let $\{e_1,\cdots, e_{r_1}\}$ be a unitary frame of $(E_{1,x},h^E|_{E_{1,x}})$ and $\{e_{r_1+1},\cdots, e_{r}\}$ be a unitary frame of $(E_{2,x},h^E|_{E_{2,x}})$. Let $(U,\{z^\alpha\}_{1\leq \alpha\leq n})$ be a local coordinate neighborhood around $x$, and denote by $E_1=U\times E_{1,x}$ the locally trivial bundle, and $\{e_i\}_{1\leq i\leq r_1}$ also gives a frame of $E_1$.
Now we define the following Hermitian metric on $E_1$ by
\begin{equation*}
 h^{E_1}( e_i,e_j):= \delta_{ij}- R_{i\b{j}\alpha\b{\beta}} z^\alpha\b{z}^\beta,\quad 1\leq i,j\leq r_1,
\end{equation*}
which is a Hermitian metric by taking $U$ small enough.
Then $(E_1,h^{E_1})$ is a Hermitian vector bundle around $x$ and satisfies
$$
  R^{E_1}_x=R^E_x|_{E_{1,x}}.
$$
Similarly, one can define a Hermitian vector bundle $(E_2,h^{E_2})$ such that 
$
  R^{E_2}_x=R^E_x|_{E_{2,x}}.
$ 
Hence
\begin{align*}
\begin{split}
  c(E_1\oplus E_2,h^{E_1}\oplus h^{E_2})|_x&=\det\left(\mr{Id}_r+\frac{\sqrt{-1}}{2\pi}\begin{pmatrix}
  R^{E_1}_x&0 \\
  0& R^{E_2}_x
\end{pmatrix}\right)\\
&=\det\left(\mr{Id}_r+\frac{\sqrt{-1}}{2\pi}\begin{pmatrix}
  R^{E}|_{E_{1,x}}&0 \\
  0& R^{E}|_{E_{2,x}}
\end{pmatrix}\right)\\
&=\det\left(\mr{Id}_r+\frac{\sqrt{-1}}{2\pi}R^E_x\right)\\
&=c(E,h^E)|_x,
 \end{split}
\end{align*}
which follows that
\begin{align*}
\begin{split}
P_\lambda(c(E,h^E))|_x=P_\lambda(c(E_1\oplus E_2,h^{E_1}\oplus h^{E_2}))|_x.
 \end{split}
\end{align*}
By Proposition \ref{prop:product}, one has 
\begin{align}\label{4.6}
\begin{split}
  P_\lambda(c(E,h^E))|_x&=\sum_{\mu,\nu}c^\lambda_{\mu\nu}P_\mu(c(E_1,h^{E_1}))|_x\wedge P_\nu(c(E_2,h^{E_2}))|_x.
 \end{split}
\end{align}
Since $(E_1,h^{E_1})$ is Nakano positive and $(E_2,h^{E_2})$ is dual Nakano positive at 
$x$, so $P_\mu(c(E_1,h^{E_1}))|_x$ and $P_\nu(c(E_2,h^{E_2}))|_x$ are positive forms. Since the Littlewood-Richardson coefficients $c^\lambda_{\mu\nu}$ are non-negative integers, see \cite[Corollary 1 in Chapter 5]{MR1464693}, so the each summand 
\begin{equation*}
  c^\lambda_{\mu\nu}P_\mu(c(E_1,h^{E_1}))|_x\wedge P_\nu(c(E_2,h^{E_2}))|_x
\end{equation*}
in RHS of \eqref{4.6} is non-negative. On the other hand, for any $\lambda,\mu,\nu$ satisfying $\lambda_i=\mu_i+\nu_i$ for all $i$, then $c^{\lambda}_{\mu,\nu}=1$, see \cite[Page 66]{MR1464693}, and $$P_\mu(c(E_1,h^{E_1}))|_x\wedge P_\nu(c(E_2,h^{E_2}))|_x$$
 is a positive $(|\lambda|,|\lambda|)$-form by Proposition \ref{product}. By \eqref{4.6}, we show that the Schur form 
$P_\lambda(c(E,h^E))|_x$ is a positive $(|\lambda|,|\lambda|)$-form. 
\begin{theorem}
	Let $(E,h^E)$ be a strongly decomposably positive vector bundle of type II over a complex manifold $X$, $\mathrm{rank} E=r$, and $\dim X=n$. Then the Schur form 
	$P_\lambda(c(E,h^E))$ is positive for any partition $\lambda\in \Lambda(k,r)$, $k\leq n$ and $k\in\mb{N}$.
\end{theorem}




\bibliographystyle{alpha}
\bibliography{Positivity}

\begin{thebibliography}{Fag22b}

\bibitem[BC65]{MR185607}
Raoul Bott and S.~S. Chern.
\newblock Hermitian vector bundles and the equidistribution of the zeroes of
  their holomorphic sections.
\newblock {\em Acta Math.}, 114:71--112, 1965.

\bibitem[BG71]{MR297773}
Spencer Bloch and David Gieseker.
\newblock The positivity of the {C}hern classes of an ample vector bundle.
\newblock {\em Invent. Math.}, 12:112--117, 1971.

\bibitem[BRT21]{MR4338228}
Sara~C. Billey, Brendon Rhoades, and Vasu Tewari.
\newblock Boolean product polynomials, {S}chur positivity, and {C}hern
  plethysm.
\newblock {\em Int. Math. Res. Not. IMRN}, (21):16636--16670, 2021.

\bibitem[Cho75]{Choi}
Man-Duen Choi.
\newblock {Completely positive linear maps on complex matrices}.
\newblock {\em Linear Algebra and its Applications}, 10(3):285--290, 1975.

\bibitem[DF22]{DF}
Simone Diverio and Filippo Fagioli.
\newblock {Pointwise universal Gysin formulae and applications towards
  Griffiths' conjecture}.
\newblock {\em The Annali della Scuola Normale Superiore di Pisa, Classe di
  Scienze}, XXIII(5):1597--1624, 2022.

\bibitem[DPS94]{MR1257325}
Jean-Pierre Demailly, Thomas Peternell, and Michael Schneider.
\newblock Compact complex manifolds with numerically effective tangent bundles.
\newblock {\em J. Algebraic Geom.}, 3(2):295--345, 1994.

\bibitem[Fag22a]{Fag20}
Filippo Fagioli.
\newblock {A note on Griffiths' conjecture about the positivity of Chern-Weil
  forms }.
\newblock {\em Differential Geometry and its Applications}, 81, 2022.
\newblock cvgmt preprint.

\bibitem[Fag22b]{Fag22}
Filippo Fagioli.
\newblock {Universal vector bundles, push-forward formulae and positivity of
  characteristic forms}.
\newblock {\em arXiv:2210.11157v1}, 2022.

\bibitem[FH91]{FH}
William Fulton and Joe Harris.
\newblock {\em Representation theory}, volume 129 of {\em Graduate Texts in
  Mathematics}.
\newblock Springer-Verlag, New York, 1991.
\newblock A first course, Readings in Mathematics.

\bibitem[Fin22]{Fin}
Siarhei Finski.
\newblock {On characteristic forms of positive vector bundles, mixed
  discriminants, and pushforward identities}.
\newblock {\em Journal of the London Mathematical Society}, 106(2):1539--1579,
  2022.

\bibitem[FL83]{FL}
William Fulton and Robert Lazarsfeld.
\newblock Positive polynomials for ample vector bundles.
\newblock {\em Ann. of Math. (2)}, 118(1):35--60, 1983.

\bibitem[Ful97]{MR1464693}
William Fulton.
\newblock {\em Young tableaux}, volume~35 of {\em London Mathematical Society
  Student Texts}.
\newblock Cambridge University Press, Cambridge, 1997.
\newblock With applications to representation theory and geometry.

\bibitem[Gri70]{Griffiths}
Phillip~A. Griffiths.
\newblock {\em Hermitian differential geometry, Chern classes, and positive
  vector bundles}, pages 185--252.
\newblock Princeton University Press, Princeton, 1970.

\bibitem[Gul12]{MR2932990}
Dincer Guler.
\newblock On {S}egre forms of positive vector bundles.
\newblock {\em Canad. Math. Bull.}, 55(1):108--113, 2012.

\bibitem[HK74]{RA}
Reese Harvey and Anthony Knapp.
\newblock {\em Positive (p,p) forms, Wirtinger's inequality, and currents},
  pages 43--62.
\newblock 01 1974.

\bibitem[Kob87]{Kobayashi+1987}
Shoshichi Kobayashi.
\newblock {\em Differential Geometry of Complex Vector Bundles}.
\newblock Princeton University Press, Princeton, 1987.

\bibitem[Li21]{MR4263677}
Ping Li.
\newblock Nonnegative {H}ermitian vector bundles and {C}hern numbers.
\newblock {\em Math. Ann.}, 380(1-2):21--41, 2021.

\bibitem[RT19]{RT}
Julius Ross and Matei Toma.
\newblock {Universal vector bundles, push-forward formulae and positivity of
  characteristic forms}.
\newblock {\em arXiv:1905.13636v3}, 2019.

\bibitem[Ste01]{MR1904379}
John~R. Stembridge.
\newblock Multiplicity-free products of {S}chur functions.
\newblock {\em Ann. Comb.}, 5(2):113--121, 2001.

\bibitem[Xia22]{MR4361967}
Jian Xiao.
\newblock On the positivity of high-degree {S}chur classes of an ample vector
  bundle.
\newblock {\em Sci. China Math.}, 65(1):51--62, 2022.

\end{thebibliography}
\end{document}